\documentclass{amsart}
\usepackage{hyperref}
\usepackage{amsthm,mathtools}
\usepackage[all]{xy}

\theoremstyle{plain}
\theoremstyle{definition}
\newtheorem{theorem}{Theorem}[section]
\newtheorem{lemma}[theorem]{Lemma}
\newtheorem{proposition}[theorem]{Proposition}
\newtheorem{corollary}[theorem]{Corollary}
\newtheorem{claim}[theorem]{Claim}

\newtheorem{example}[theorem]{Example}

\newtheorem{remark}[theorem]{Remark}

\usepackage{amssymb,amsmath,tabularx,graphicx,float,placeins}
\usepackage{tikz}
\usetikzlibrary[calc,intersections,through,backgrounds,arrows,decorations.pathmorphing]

\def\Cbb{\mathbb{C}}\def\Dbb{\mathbb{D}}\def\Ebb{\mathbb{E}}\def\Zbb{\mathbb{Z}}

\def\Sfrak{\mathfrak{S}}

\def\ov{\overline}\def\pr{\prime}\def\ra{\rightarrow}\def\setm{\setminus}\def\sqci{\raisebox{1.0pt}\fullmoon\square}\def\un{\underline}

\DeclareMathOperator{\wt}{wt}

\usepackage{wasysym}

\begin{document}

\title{Roots of descent polynomials and an algebraic inequality on hook lengths}
\author{Pakawut Jiradilok and Thomas McConville}
\date{\today}

%\author{Pakawut Jiradilok}
\address{Department of Mathematics, Massachusetts Institute of Technology, Cambridge, Massachusetts}
\email[P.~Jiradilok]{pakawut@mit.edu}

%\author{Thomas McConville}
\address{Department of Mathematics, Kennesaw State University, Marietta, Georgia}
\email[T.~McConville]{tmcconvi@kennesaw.edu}

\begin{abstract}
We prove a conjecture by Diaz-Lopez et al. that bounds the roots of descent polynomials. To do so, we prove an algebraic inequality, which we refer to as the ``Slice and Push Inequality.'' This inequality compares expressions that come from Naruse's hook-length formula for the number of standard Young tableaux of a skew shape.
\end{abstract}

\maketitle

\section{Introduction}

In 2014, Naruse \cite{naruse2014schubert} announced a remarkable formula for $f^{\lambda/\mu}$, the number of standard Young tableaux of skew shape $\lambda/\mu$. Later known as {\em Naruse's (hook-length) formula} in the literature, the formula expresses $f^{\lambda/\mu}$ as a sum over combinatorial objects called {\em excited (Young) diagrams}. In the context of equivariant Schubert calculus, Ikeda and Naruse \cite{ikeda2009excited} introduced these excited diagrams a few years before Naruse's discovery of the skew-shape hook-length formula.

Since the inception of the celebrated hook-length formula by Frame-Robinson-Thrall \cite{frame1954hook} in 1954, many have studied, re-proved, and generalized the formula. In the same manner, since Naruse's discovery, many combinatorialists have been investigating Naruse's formula in recent years. Notably, Morales, Pak, and Panova have developed a series of papers studying the formula, in which new proofs, $q$-analogues, and many new properties of Naruse's formula have been presented; see e.g. \cite{morales2015hook}. Konvalinka \cite{konvalinka2017bijective} has also given a bijective proof of Naruse's formula. It is worth mentioning that, in addition to the hook-length formula for skew straight shapes, Naruse \cite{naruse2014schubert} also announced formulae for skew shifted shapes (types B and D), and for these, Konvalinka \cite{konvalinka2018hook} gave bijective proofs as well.

For us, one of the main advantageous attributes of Naruse's hook-length formula is that it is cancellation-free. In \cite{morales.pak.panova:2018asymptotics}, Morales, Pak, and Panova have exploited this positive sum property of Naruse's formula to establish intriguing asymptotic bounds on the number of standard Young tableaux of skew shapes. Combining the variational principle and Naruse's formula, Morales, Pak, and Tassy \cite{morales.pak.tassy:2018asymptotics2} prove fascinatingly precise limiting behaviors of the number of standard Young tableaux of shew shapes, proving and generalizing conjectures in \cite{morales.pak.panova:2018asymptotics}. From this point of view, Naruse's formula provides an efficient tool to develop algebraic inequalities related to combinatorial objects.

In this note, we present combinatorial objects which we call ``$\sqci$-diagrams.'' These objects are closely related to the excited diagrams in Naruse's hook-length formula. By exploiting the cancellation-free property of Naruse's formula, we introduce and prove the Slice and Push Inequality, which is an algebraic inequality on $\sqci$-diagrams. Using the Slice and Push Inequality, we provide bounds on the roots of descent polynomials, thus proving a conjecture of Diaz-Lopez, Harris, Insko, Omar, and Sagan \cite{diaz2017descent}, which we recall below.

Let $\Sfrak_n$ be the set of permutations of $[n]:=\{1,\ldots,n\}$. For a permutation $\pi=\pi_1\cdots\pi_n$, the \emph{descent set} of $\pi$ is the set of positions $i\in[n-1]$ such that $\pi_i>\pi_{i+1}$. Given a finite set of positive integers $I\subseteq\Zbb_{>0}$, the \emph{descent polynomial} $d_I(z)$ is the unique polynomial such that $d_I(n)$ is the number of elements of $\Sfrak_n$ whose descent set is $I$ (assuming $n>\max(I \cup \{0\})$). MacMahon introduced the descent polynomial in \cite{macmahon2001combinatory}, where polynomiality of this function was proved by an inclusion-exclusion argument. One can also deduce this from Naruse's formula, as we show in Section~\ref{subsec:ribbons}.

Using a recurrence for descent polynomials, \cite{diaz2017descent} proved that $d_I(z)$ is a degree $m=\max(I\cup\{0\})$ polynomial with $d_I(i)=0$ for all $i\in I$. We let $|z|,\ \Re(z),$ and $\Im(z)$ denote the complex modulus, the real part, and the imaginary part of $z$, respectively. Diaz-Lopez et al. \cite{diaz2017descent} conjectured the following bounds on the roots of the descent polynomial.

\begin{theorem}[\cite{diaz2017descent}, Conjecture 4.3]\label{conj:main}
  If $z_0$ is a complex number such that $d_I(z_0)=0$, then
  \begin{enumerate}
  \item\label{conj:main1} $|z_0|\leq m$, and
  \item\label{conj:main2} $\Re(z_0)\geq -1$.
  \end{enumerate}
\end{theorem}

It is known that these bounds are optimal, though our data suggest that the (convex) region bounded by these two inequalities is much larger than necessary for large $m$. Stronger inequalities in the special case where $I=\{m\}$ were proved in \cite[Theorem 4.7, Corollary 4.8]{diaz2017descent}. A plot of the roots of the descent polynomial $d_{\{10\}}$ is given in Figure~\ref{plot}. The shaded region is determined by the inequalities in Theorem~\ref{conj:main}.

\begin{figure}
\begin{tikzpicture}[scale = 0.5]
\filldraw[fill=cyan!50, draw = cyan!50] (0,0) circle (10);
\filldraw[fill = white, draw = white] (-10,-10) rectangle (-1,10);
\draw[step=1cm,gray,very thin] (-5,-11) grid (15,11);
\draw[thick,->] (0,0) -- (15,0);
\draw[thick,->] (0,0) -- (-5,0);
\draw[thick,->] (0,0) -- (0,11);
\draw[thick,->] (0,0) -- (0,-11);
\filldraw[fill=orange, draw = orange] (-1,0) circle (0.2);
\filldraw[fill=orange, draw = orange] (0.123556,-2.080159) circle (0.2);
\filldraw[fill=orange, draw = orange] (0.123556,2.080159) circle (0.2);
\filldraw[fill=orange, draw = orange] (2.852828,-3.469780) circle (0.2);
\filldraw[fill=orange, draw = orange] (2.852828,3.469780) circle (0.2);
\filldraw[fill=orange, draw = orange] (6.147172,-3.469780) circle (0.2);
\filldraw[fill=orange, draw = orange] (6.147172,3.469780) circle (0.2);
\filldraw[fill=orange, draw = orange] (8.876444,-2.080159) circle (0.2);
\filldraw[fill=orange, draw = orange] (8.876444,2.080159) circle (0.2);
\filldraw[fill=orange, draw = orange] (10,0) circle (0.2);

\node [rotate=90] at (-6, 0) {imaginary part};
\node at (4.5, -12) {real part};
\node at (18,0) {};
\end{tikzpicture}
\caption{\label{plot}A plot of the roots of $d_{\{10\}}(z)$}
\end{figure}
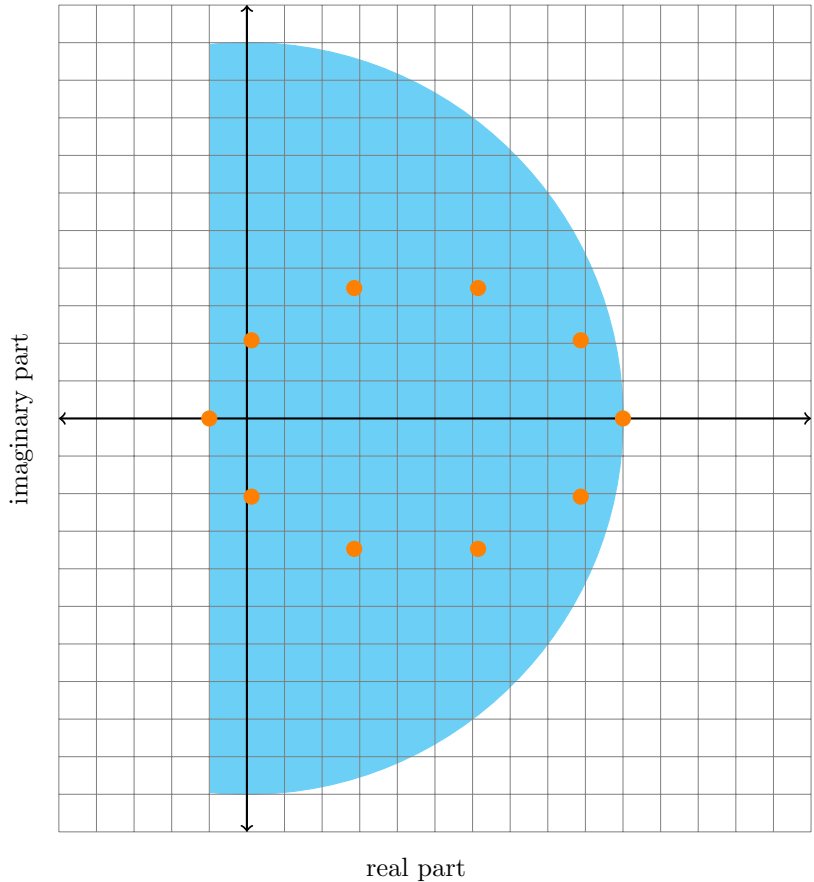

Our main result is a proof of Theorem~\ref{conj:main}. Our proof relies on a reinterpretation of descent polynomials as functions enumerating standard Young tableaux of a family of skew shapes known as \emph{ribbons}. Background on Naruse's formula and the connection to descent polynomials are given in Section~\ref{sec:naruse}. Using this translation, we replace the descent polynomial $d_I(z)$ with a polynomial $E(z)$ so that if $d_I(z_0)=0$ for some complex number $z_0$ then either $z_0\in I$ or $E(z_0-m)=0$. The proof of Theorem~\ref{conj:main} relies on Proposition~\ref{prop:ineq_coeff}, which is an inequality of the coefficients of $E(z)$ in a nonstandard basis for degree $\leq m$ polynomials. The proof of Proposition~\ref{prop:ineq_coeff} is given as a consequence of an algebraic inequality on the weights of $\sqci$-diagrams, defined in Section~\ref{sec:slice_push}. We refer to this inequality as the ``Slice and Push Inequality'' as it involves dividing and shifting cells of a Young diagram. The proof of Theorem~\ref{conj:main} is wrapped up in Section~\ref{sec:main_conj}, which relies on a couple of technical results on roots of polynomials proved in the appendices.

Independently, Bencs \cite[Theorem 5.2]{bencs2018some} discovered a separate proof of Theorem~\ref{conj:main}(\ref{conj:main1}) and proved Theorem~\ref{conj:main}(\ref{conj:main2}) for most choices of $I$; see \cite[Section 6]{bencs2018some}. His proofs rely on some inequalities satisfied by the coefficients of $d_I(z)$ in different bases for polynomials of degree $\leq m$ than the ones we consider.

Besides our main result -- a resolution of the conjecture of Diaz-Lopez et al. on the location of roots -- various exact bounds in this note are intriguing, and are subject of further investigation in their own right. Let us briefly describe the Slice and Push Inequality here. The excitation factor from Naruse's hook-length formula (cf. Section \ref{sec:naruse} for precise details) is a sum over all excited diagrams of the weight of each excited diagram. The weight of our $\sqci$-diagram is similar to the excitation factor from Naruse's formula. Instead of allowing each cell to be excited many times as in Naruse's excitation, our $\sqci$-diagram allows each $\raisebox{1.0pt}\fullmoon$-type cell to be excited {\em at most once}, and disallows any movement of the $\square$-type cell. Similar to Naruse's excitation factor, this $\sqci$ rule produces a collection of diagrams, and then we sum the weights of the diagrams over this collection to obtain the weight of a $\sqci$-diagram (cf. the beginning of Section \ref{sec:slice_push} for a precise description). The Slice and Push Inequality says that if we start with a $\sqci$-diagram, consider any vertical line between any two consecutive columns, push all the $\raisebox{1.0pt}\fullmoon$-type cells to the right of the line one unit to the right, change all the pushed cells to $\square$, and obtain a new $\sqci$-diagram, then the weight of the original diagram is greater than or equal to the weight of the new diagram. In the notation of Lemma \ref{lem:s_and_p}, this bound can be written concisely:
\[
\wt(D; F) \geq \wt(D|_{k};\ F\sqcup (|_{k}D)^{\ra}).
\]
See Figure \ref{fig:sap1} for a depiction of the Slice and Push Inequality. Also, see Example \ref{ex:sap_55432} for an explicit calculation.

Another intriguing result is the exact bound in Proposition \ref{prop:ineq_coeff}. It turns out that under the Newton basis coming from Naruse's formula, the coefficients $C_0, C_1, \ldots, C_s$ of the descent polynomial ``do not grow too quickly'' so that the sequence
\[
\left( \frac{C_0}{0!}, \, \frac{C_1}{1!}, \, \ldots \, , \, \frac{C_s}{s!} \right)
\]
is weakly decreasing. These numbers $C_0, C_1, \ldots, C_s$ are defined in Section \ref{sec:naruse}. Proposition \ref{prop:ineq_coeff} is an application of the Slice and Push Inequality, and the proposition is in turn a key ingredient in the proof of the conjecture of Diaz-Lopez et al.

Aside from the Slice and Push Inequality and the bound on the growth of $C_0, C_1, \ldots, C_s$, we present further interesting bounds in the two appendices. These results are less combinatorial, and more analytic. They are also ingredients in the proof of our main theorem. The reader can enjoy Appendices \ref{app:analytic} and \ref{app:perturbation} separately from the rest of the paper. We think the exact bounds we present there are, once again, intriguing, and we would be interested to see further applications of these inequalities to other combinatorial problems.

We remark that there is a related family of polynomials, the \emph{peak polynomials} $P_I(z)$, defined by the property that $2^{n-|I|-1}P_I(n)$ is the number of permutations of $\Sfrak_n$ with peak set $I$. Here, the \emph{peak set} of a permutation $\pi=\pi_1\cdots\pi_n$ is the set of positive integers $i\in\{2,\ldots,n-1\}$ such that $\pi_{i-1}<\pi_i>\pi_{i+1}$. Polynomiality of this function was proved in \cite{billey2013permutations}. It has been observed that peak polynomials and descent polynomials have many similar properties. In particular, \cite[Conjecture 4.3]{diaz2017descent} was motivated by a similar conjecture bounding the roots of peak polynomials \cite[Conjecture 1.6]{billey2016coefficients}. Supporting the connection, O\u{g}uz \cite{ouguz2018connecting} proved that descent polynomials can be expressed as a sum of peak polynomials, and conversely, each peak polynomial is an alternating sum of descent polynomials. Our approach to bounding the roots of descent polynomials does not seem to have a clear application to peak polynomials, however. In particular, we would be interested in finding a basis for polynomials of small degree for which the coefficients of the peak polynomial $P_I(z)$ satisfy the conditions of the Polynomial Perturbation Lemma in Appendix~\ref{app:perturbation}.

This paper is organized as follows. In Section \ref{sec:naruse}, we review Naruse's hook-length formula. We establish a connection between Naruse's formula and the descent polynomial. In Section \ref{sec:slice_push}, we investigate the Slice and Push Inequality. We also prove Proposition \ref{prop:ineq_coeff}. In Section \ref{sec:main_conj}, we give a proof of the conjecture of Diaz-Lopez, Harris, Insko, Omar, and Sagan. Some technical analytic lemmas used in the proof of the main conjecture are proved in Appendices \ref{app:analytic} and \ref{app:perturbation}.

\section{Naruse's formula}\label{sec:naruse}

\subsection{Excited diagrams}\label{subsec:excited}

A partition $\lambda=(\lambda_1\geq\lambda_2\geq \cdots)$ is a weakly decreasing sequence of nonnegative integers such that $\lim_{i \rightarrow \infty} \lambda_i = 0$. Its conjugate partition is denoted $\lambda^{\pr}=(\lambda_1^{\pr},\lambda_2^{\pr},\ldots)$. A \emph{Young diagram} $\Dbb(\lambda)$ for a partition $\lambda$ is a left-justified array of cells, with $\lambda_1$ cells in the first (top) row, $\lambda_2$ cells in the next row, and so on. For a diagram $\Dbb$, we let $c_{i,j}$ be the cell in the $i$-th row and the $j$-th column. If $\mu\subseteq\lambda$, we draw the \emph{skew Young diagram} for $\lambda/\mu$ by shading in the cells contained in $\mu$. When considering a fixed skew shape $\lambda/\mu$, we typically let $n$ be the \emph{size} of the shape; i.e.,
\[ n = |\lambda| - |\mu| = \sum_{i=1}^{\infty} (\lambda_i-\mu_i). \]
A \emph{standard Young tableau} of shape $\lambda/\mu$ is a bijective filling of the unshaded cells of $\lambda$ with numbers $\{1,\ldots,n\}$ such that values increase to the right along any row and going down along any column. The six standard tableaux of shape $(3,3,3)/(2,2)$ are shown in Figure~\ref{fig:tab} (left).

\begin{figure}
  
  \centering
  \includegraphics{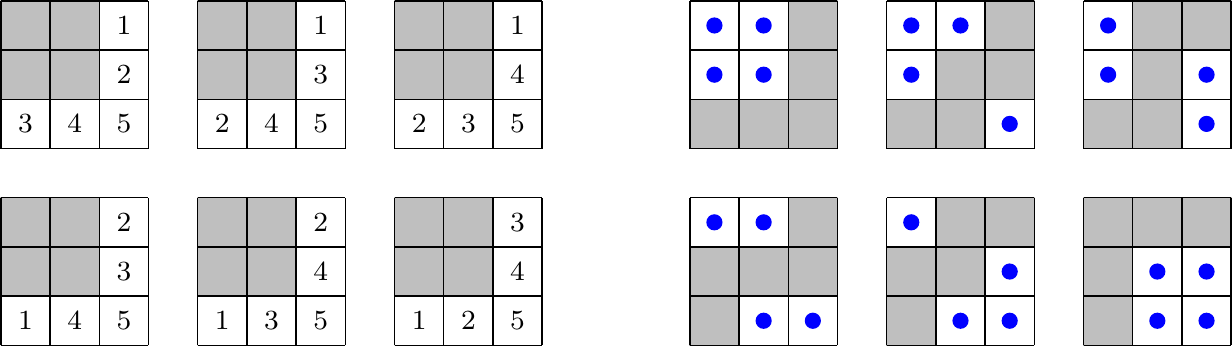}
  \caption{\label{fig:tab}(left) Standard tableaux  \hspace{3mm} (right) Excited diagrams}
  
\end{figure}

Let $f^{\lambda/\mu}$ be the number of standard Young tableaux of shape $\lambda/\mu$, setting $f^{\lambda/\mu}=f^{\lambda}$ if $\mu=\ov{0}:=(0,0,\ldots)$. For a cell $c\in\Dbb(\lambda)$, its \emph{hook length} $h(c)=h(c;\lambda)$ is the number of cells that lie in the same row, weakly to the right of $c$, or in the same column, strictly below $c$. The Frame-Robinson-Thrall ``hook length formula'' is a remarkable product formula for the number of standard Young tableaux of a (non-skew) shape:
\[ f^{\lambda} = |\lambda|!\prod_{c\in\Dbb(\lambda)}\frac{1}{h(c)}. \]
For some skew shapes $\lambda/\mu$, the number $f^{\lambda/\mu}$ has large prime factors, removing the possibility for such a simple product formula in general. We refer to the recent survey~\cite{adin2015standard} for a wide array of formulae for $f^{\lambda/\mu}$.

We use a formula recently discovered by Naruse~\cite{naruse2014schubert}, recalled in Theorem~\ref{thm:naruse}. Fix a skew shape $\lambda/\mu$. Divide the cells of $\Dbb(\lambda)$ into collections $\ldots,X_{-1},X_0,X_1,\ldots$ according to their \emph{contents}; that is,
\[ X_k = \{ c_{i,j}\in \Dbb(\lambda)\ |\ k=j-i \}. \]
We consider cells to be partially ordered so that $c\leq c^{\pr}$ if $c$ is weakly northwest of $c^{\pr}$; that is, $c_{i,j}\leq c_{i^{\pr},j^{\pr}}$ if $i\leq i^{\pr}$ and $j\leq j^{\pr}$. An \emph{excited diagram} $D$ of type $\lambda/\mu$ is a subset of cells of $\Dbb(\lambda)$ such that
\begin{itemize}
  \item there exists a bijection $\eta:\Dbb(\mu)\ra D$ with $\eta(c_{i,j})\in X_{j-i}$ for all $c_{i,j}\in\Dbb(\mu)$, and
  \item for each $k \in \mathbb{Z}$, the restriction of $\eta$ to $\Dbb(\mu)\cap(X_k\cup X_{k+1})$ is order-preserving.
\end{itemize}
We let $\Ebb(\lambda/\mu)$ be the set of excited diagrams of type $\lambda/\mu$. The six excited diagrams of shape $(3,3,3)/(2,2)$ are shown in Figure~\ref{fig:tab} (right).

\begin{theorem}[Naruse]\label{thm:naruse}
  For a skew shape $\lambda/\mu$ of size $n$,
  \[ f^{\lambda/\mu} = \frac{n!}{\prod_{c\in\Dbb(\lambda)}h(c)}\sum_{D\in\Ebb(\lambda/\mu)}\prod_{c^{\pr}\in D}h(c^{\pr}). \]
\end{theorem}

Applying this formula to the shape $(3,3,3)/(2,2)$, we get the identity
\[ 6 = \frac{5!}{5\cdot 4^2\cdot 3^3\cdot 2^2\cdot 1}\left(5\cdot 4^2\cdot 3 + 5\cdot 4^2 \cdot 1 + 5\cdot 4\cdot 2 \cdot 1 + 5\cdot 4\cdot 2 \cdot 1 + 5\cdot 2^2 \cdot 1 + 3\cdot 2^2 \cdot 1\right). \]

Theorem~\ref{thm:naruse} was discovered by Naruse and announced in \cite{naruse2014schubert}. For (several) proofs of the formula and pointers to the literature, we recommend \cite{morales2015hook}.

\subsection{Skew shapes with varying first row}\label{subsec:polynomials}

For $t\in\Zbb_{>0}$, we let $\lambda^{(t)}$ be the partition obtained from $\lambda$ by replacing the first part $\lambda_1$ with $\lambda_1+t-1$. For a fixed shape $\lambda/\mu$, we define the size of $\lambda^{(t)}/\mu$ to be $n+t-1$ so that $n$ does not depend on $t$. We consider the function
\[ p(t;\lambda/\mu) = f^{\lambda^{(t)}/\mu}. \]
If the skew shape $\lambda/\mu$ is understood, we simply write $p(t)$ for this function. For the remainder of this section, we will assume that $\lambda_1=\lambda_2$. We are free to make this assumption since $p(t;\lambda^{(u)}/\mu)=p(t+u-1;\lambda/\mu)$ for $u\geq 1$. We fix some additional parameters:
\begin{itemize}
\item $r=\lambda_1$,
\item $s=\mu_1$, and
\item $\alpha_i=h(c_{1,i}; \lambda)-1$ for $1\le i\le r$.
\end{itemize}

We use Theorem~\ref{thm:naruse} to give a nice formula for $p(t)$ in Lemma~\ref{lem:poly_factorization}. Before doing so, we first make a few observations.

The only cells whose hook lengths vary with $t$ are those in the first row. For each excited diagram $D \in \mathbb{E}(\lambda^{(t)}/\mu)$, let $\ov{D}$ be the subdiagram of $D$ obtained by removing all cells from the first row of $\Dbb(\lambda^{(t)})$.

Observe that the shapes $\lambda^{(t)}/\mu$ and $\lambda/\mu$ have the \emph{same} set of excited diagrams $\Ebb(\lambda/\mu)$. Furthermore, if $D$ is an excited diagram with $d$ cells in the first row, then its first row must be $\{c_{1,1},\ldots,c_{1,d}\}$. We let $E(t)$ be the \emph{excitation factor} of $\lambda^{(t)}/\mu$; i.e.,
\begin{align*}
  E(t) &= \sum_{D\in\Ebb(\lambda^{(t)}/\mu)}\left(\prod_{c\in \ov{D}}h(c)\right)\left(\prod_{c_{1,i}\in D}(t+\alpha_i)\right) \\
  &= \sum_{d=0}^s\left( \prod_{i=1}^d(t+\alpha_i)\right) \left( \sum_{\substack{D\in\Ebb(\lambda/\mu)\\\#(D - \ov{D}) = d}} \prod_{c\in \ov{D}}h(c) \right)
\end{align*}

When $d=s$, the inner sum is nonempty, so $E(t)$ has degree {\em exactly} $s$. Let $C_0,C_1,\ldots,C_s$ be the nonnegative integers for which
\[ E(t) = C_0(t+\alpha_1)\cdots(t+\alpha_s) + C_1(t+\alpha_1)\cdots(t+\alpha_{s-1}) + \cdots + C_{s-1}(t+\alpha_1) + C_s \] 
so that $C_d(t+\alpha_1)\cdots(t+\alpha_{s-d})$ is the ``contribution'' to the excitation factor from those excited diagrams with $s-d$ cells in the first row.

\begin{remark} \label{rmk:NB}
The list of polynomials $1,\ t+\alpha_1,\ (t+\alpha_1)(t+\alpha_2),\dots,\ (t+\alpha_1) \cdots (t+\alpha_s)$ is an example of a Newton basis for the space of polynomials of degree $\leq s$. A sequence of polynomials $(p_k(t))_{k=0}^s$ is a \emph{Newton basis} if there exist complex numbers $\beta_1,\ldots,\beta_s\in\Cbb$ and $\lambda_0,\ldots,\lambda_s\in\Cbb\setm\{0\}$ such that $p_k(t)=\lambda_k\prod_{i=1}^k(t+\beta_i)$, for $k = 0, 1, \dots, s$. Using inequalities on the coefficients of polynomials with respect to a Newton basis is a common approach to proving bounds on the roots of those polynomials. This approach is taken for bounding the roots of descent polynomials in \cite{bencs2018some} and \cite{diaz2017descent} using the falling factorial basis.
\end{remark}

\begin{remark}
After the preprint version of this present note became available, Cai \cite{cai2021ratios} investigated this sequence $(C_0, C_1, \ldots, C_s)$ of coefficients. Cai calls these coefficients ``Naruse-Newton coefficients,'' and examines their various properties, such as log-concavity, unimodality, and limiting behavior. In Appendix A of \cite{cai2021ratios}, Cai gives tables listing the values of the coefficients for all nonempty descent sets $I \subseteq \{1, 2, \ldots, 7\}$.
\end{remark}

We now calculate
\begin{equation} \label{eq:pt}
  p(t) = (n+t-1)!\left(\prod_{c\in\ov{\Dbb(\lambda)}}\frac{1}{h(c)}\right)\left(\prod_{i=1}^{r}\frac{1}{t+\alpha_i}\right)\frac{1}{(t-1)!}E(t).
\end{equation}

For convenience, we will assume that $\Dbb(\lambda/\mu)$ is connected; i.e., $\mu_i<\lambda_{i+1}$ whenever $\mu_i\neq 0$. In particular, this implies that $\alpha_1\leq n-1$ holds. By canceling common factors in the expression above, we obtain a useful factorization of the polynomial $p(t)$ in Lemma~\ref{lem:poly_factorization}. We collect some of the factors that do not depend on $t$ as a constant $C$. The first, the third, and the fourth factors combine to be a polynomial in $t$. 

\begin{lemma}\label{lem:poly_factorization}
  If $\Dbb(\lambda/\mu)$ is connected, then there exists a positive real number $C$ not depending on $t$ such that
  \[ p(t) = C\cdot E(t)\prod_{\substack{\beta\in\{0,1,\dots,n-1\}\\\beta\notin\{\alpha_1,\ldots,\alpha_r\}}}(t+\beta). \]
\end{lemma}

Lemma~\ref{lem:poly_factorization} allows us to reduce the problem of bounding the roots of $p(t)$ to bounding the roots of the lower degree polynomial $E(t)$, which we do in Section~\ref{sec:main_conj}.

\subsection{Ribbons and descent polynomials}\label{subsec:ribbons}

A \emph{ribbon} is a (nonempty) connected skew Young diagram $\Dbb$ that does not contain a $2\times 2$ block of cells. If $T$ is a standard filling of $\Dbb$, then $T$ determines a permutation $\pi(T)=\pi_1\cdots\pi_n$ whose entries appear in order along the ribbon, starting from the bottom left corner to the upper right. The positions of the descents of $\pi(T)$ are determined by the shape of $\Dbb$, as illustrated in Figure~\ref{fig:ribbon}, where the shape of the ribbon forces the permutation $\pi(T)$ to have descent set $I=\{3,5,8,9,11\}$. Namely, there is a descent at $i$ if and only if the $i$-th cell of the ribbon is below the $(i+1)$-st cell. Conversely, if $I\subseteq[n-1]$ we may construct a ribbon $\Dbb$ for which the permutations $\pi$ with descent set $I$ are of the form $\pi=\pi(T)$ for some standard filling $T$ of $\Dbb$. Recall that when $I$ is a descent set, we let $m$ denote $\max(I \cup \{0\})$.

\begin{figure}
  
  \centering
  \includegraphics{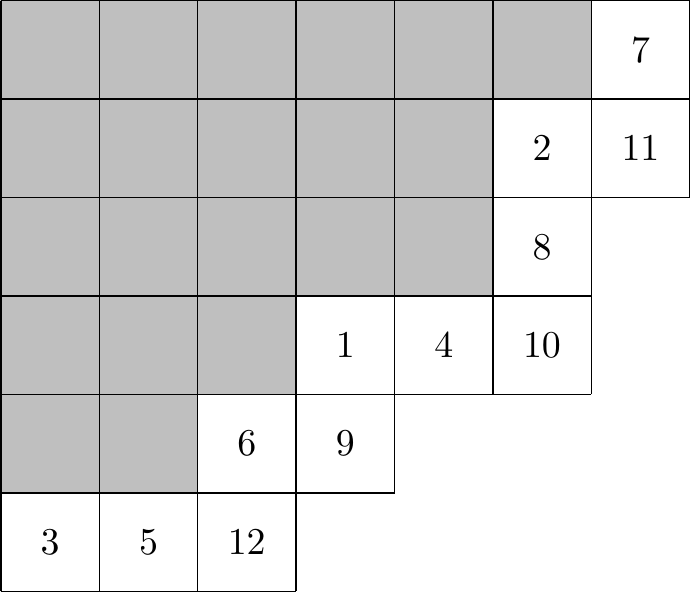}
  \caption{\label{fig:ribbon}A tableau of ribbon shape}
    
\end{figure}

Now suppose $\Dbb(\lambda/\mu)$ is a ribbon. Combined with the assumption that $\lambda_1=\lambda_2$, this implies $\mu_1=\lambda_1-1$. Hence, $s=r-1,\ m=n-1$, and the polynomial $p(t)$ has degree $m$.

In this dictionary between standard fillings of a ribbon and permutations with a given descent set, the addition of cells to the first row of $\Dbb$ corresponds to taking longer permutations without changing the descent set. So if $\Dbb(\lambda/\mu)$ is a ribbon shape corresponding to the descent set $I$, we have the identity of polynomials $p(t-m)=d_I(t)$. Furthermore, one may observe that the ascents in the permutation are either in $\{m+1,m+2,\ldots\}$ or $\{m-\alpha_i:\ i\in\{2,\ldots,r\}\}$.

Using Lemma~\ref{lem:poly_factorization}, we see that an integer $\gamma\in\{-m,\ldots,0\}$ is among the roots of $p(t)$ whenever $\gamma\neq-\alpha_i$ for any $i\in[r]$. This means $\gamma$ is a root of $p(t)$ whenever $m+\gamma$ is in $I$. These are precisely the roots of the descent polynomials indicated in \cite[Theorem 4.1]{diaz2017descent}.

To prove suitable bounds on the roots of the excitation factor $E(t)$, we show that the sequence of coefficients $(C_0,C_1,\ldots,C_s)$ does not ``grow too quickly,'' in the following sense.

\begin{proposition}\label{prop:ineq_coeff}
  For ribbons, with $(C_0,\ldots,C_s)$ defined as above,
  \[ \frac{C_0}{0!}\geq \frac{C_1}{1!}\geq \frac{C_2}{2!}\geq\cdots\geq \frac{C_s}{s!}. \]
\end{proposition}

The proof of Proposition~\ref{prop:ineq_coeff} will be given in Section~\ref{sec:slice_push}. These inequalities are used in Section~\ref{sec:main_conj} to prove Theorem~\ref{conj:main}.

Let us add a few remarks here. First, observe that for any descent set $I$, when we fill in the hook lengths inside the cells of the Young diagram of $\lambda^{(t)}$, the hook lengths that appear in the cells on the first column but strictly below the first row are exactly the elements of $I$. See, for example, Figure \ref{d35}, where the descent set is $I = \{3,5\}$. Note the hook lengths of $5$ and $3$ on the first column. Second, Naruse's formula is impressive as it expresses the number of skew tableaux in terms of a {\em positive} sum over combinatorial objects. By this property, we know immediately that if we expand any descent polynomial in the Newton basis (as described in Remark \ref{rmk:NB}), the coefficients $C_0, C_1, \ldots, C_s$ are {\em strictly positive integers}. Note that many algebraic combinatorics problems involve attempting to give some positive integers discovered elsewhere an elegant combinatorial interpretation. Here, given the formula of Naruse's, these coefficients come with their combinatorial meanings by construction. Third, it is known that the descent polynomial $d(I; x)$ (using the notation of \cite{diaz2017descent}) may have some negative integer coefficients in general, when we expand in the usual polynomial basis $1, x, x^2, \ldots$. By the analysis using Naruse's formula we show here, we obtain as an immediate consequence that for any nonempty descent set $I$, all the coefficients of $d(I; x+\alpha_1)$ in the usual polynomial basis $1, x, x^2, \ldots$ are strictly positive. (Recall that when $I$ is nonempty, $\alpha_1 = m = \max(I)$.)

\subsection{Computing the descent polynomial via Naruse's formula}\label{subsec:computing}

Before ending Section \ref{sec:naruse}, we use this subsection here to describe explicitly how one uses Naruse's hook-length formula to compute the descent polynomial easily. Example \ref{ex:d35} showcases an explicit calculation in the case $I = \{3,5\}$.

Fix a descent set $I$. From the Young diagram of the corresponding skew shape $\lambda/\mu$, we determine $\alpha_1, \alpha_2, \ldots, \alpha_s$. From Equation (\ref{eq:pt}) in Subsection \ref{subsec:polynomials}, we write the descent polynomial $p(t)$ as a product
\[
p(t) = T(t) \cdot E(t),
\]
where $T(t)$ is the {\em trivial} part, and $E(t)$ is the excitation factor. The trivial part has an easy description:
\[
T(t) = \frac{1}{\text{product of hook lengths below the first row of } \lambda} \cdot \prod_{i \in I} (t+\alpha_1 - i).
\]
The excitation factor $E(t)$ is computed using the combinatorial formula in Subsection \ref{subsec:polynomials}.

We remark that our notation $p(t)$ for the descent polynomial is slightly different from the notation $d(I; - )$, used in \cite{diaz2017descent}. To compute explicitly the polynomial $d(I; N)$, the number of permutations in $\mathfrak{S}_N$ whose descent sets are exactly $I$, we simply note the translation $N = t+\alpha_1$. As a result, we have
\[
d(I; N) = p(N - \alpha_1).
\]
Using our formula for $p(t)$ above, we obtain the desired descent polynomial.

One interesting immediate consequence of the expression $p(t) = T(t) \cdot E(t)$ is that we may recover Lemma 3.8 of \cite{diaz2017descent}. Observe that $d(I;0) = p(-\alpha_1) = T(-\alpha_1) \cdot E(-\alpha_1)$. We know that
\begin{align*}
T(-\alpha_1) &= \frac{1}{\text{product of hook lengths below the first row of } \lambda} \cdot \prod_{i \in I} (-i) \\
&= \frac{(-1)^{\# I} \cdot \prod_{i \in I} i }{\text{product of hook lengths below the first row of } \lambda}
\end{align*}
On the other hand, $E(-\alpha_1) = C_s$, which is the product of hook lengths below the first row, strictly to the right of the first column, of $\lambda$. Since the hook lengths below the first row on the first column of $\lambda$ are exactly the elements of $I$, we find that
\[
E(-\alpha_1) = \frac{\text{product of hook lengths below the first row of } \lambda}{\prod_{i \in I} i },
\]
whence $d(I;0) = (-1)^{\# I}$, which is precisely Lemma 3.8 of \cite{diaz2017descent}.

\begin{figure}
\begin{tikzpicture}[scale = 7/8]
\def\w{3/2} % width of a box.
\def\h{1} % height of a box.
\def\r{0.3} % radius of a circle.
\def\s{5/8} % side length of a square.
\def\tw{\w/2} % to shift the text to the right.
\def\th{\h/2} % to shift the text up.

% Code to draw a circle.
\def\drawacircle[#1,#2,#3]{
\filldraw[#3] (#1 +\w/2, #2 + \h/2) circle (\r);
}

% Code to draw a square.
\def\drawasquare[#1,#2,#3]{
\filldraw[#3] (#1 +\w/2 -\s/2, #2 + \h/2 -\s/2) rectangle (#1 +\w/2 +\s/2, #2 + \h/2 +\s/2);
}

% Code to write some text in a cell.
\def\addtext[#1,#2,#3]{
\node at (#1*\w + \tw,#2*\h + \th) {#3};
}

% Shade cells in $\mu$.
\filldraw[gray!15] (0*\w,1*\h) rectangle (2*\w,3*\h);
\filldraw[gray!15] (2*\w,2*\h) rectangle (3*\w,3*\h);

% Draw the Young diagram.
\foreach \xa\ya\xb\yb in {0/0/3/3,3/1/4/3,4/2/9/3}
{
  \draw[xstep = \w, ystep = \h] (\xa*\w,\ya*\h) grid (\xb*\w,\yb*\h);
}

% Add texts.
\addtext[0,2,$t+5$]
\addtext[1,2,$t+4$]
\addtext[2,2,$t+3$]
\addtext[3,2,$t+1$]
\addtext[4,2,$t-1$]
\addtext[5,2,$t-2$]
\addtext[6,2,$\cdots$]
\addtext[7,2,$2$]
\addtext[8,2,$1$]
\addtext[0,1,$5$]
\addtext[1,1,$4$]
\addtext[2,1,$3$]
\addtext[3,1,$1$]
\addtext[0,0,$3$]
\addtext[1,0,$2$]
\addtext[2,0,$1$]

\end{tikzpicture}
\caption{\label{d35} The hook lengths for $\lambda^{(t)}$ when $\lambda = (4,4,3)$.}
\end{figure}

\begin{example} \label{ex:d35}
In this example, we compute the descent polynomial for the descent set $I = \{3,5\}$. Take a look at Figure \ref{d35}. The ribbon corresponding to $I = \{3,5\}$ is $\lambda/\mu$, where $\lambda = (4,4,3)$ and $\mu = (3,2)$. In the figure, the partition $\lambda^{(t)}$ is shown along with the hook length in each cell. The five cells of $\mu$ are shaded. In this case, we have $r = 4$ and $s = 3$. The $\boldsymbol\alpha$-vector is $\boldsymbol\alpha = (\alpha_1, \alpha_2, \alpha_3, \alpha_4) = (5,4,3,1)$. In general, the hook lengths on the first row of $\lambda^{(t)}$ are always
\[
t+\alpha_1, t+\alpha_2, \ldots, t+ \alpha_r, t-1, t-2, \ldots, 3, 2, 1.
\]

There are $9$ excited diagrams of $\lambda/\mu$. Three of them are scalar multiples of $(t+5)(t+4)(t+3)$. Another three are multiples of $(t+5)(t+4)$. Two are multiples of $(t+5)$. The other one is a scalar. We write
\[
E(t) = C_0(t+5)(t+4)(t+3) + C_1(t+5)(t+4) + C_2(t+5) + C_3.
\]
These coefficients can be computed directly:
\begin{align*}
C_0 &= 5 \cdot 4 + 5 \cdot 1 + 2 \cdot 1 = 27, \\
C_1 &= (5 \cdot 4 + 5 \cdot 1 + 2 \cdot 1) \cdot 1 = 27, \\
C_2 &= (5+2) \cdot 3 \cdot 1 \cdot 1 = 21, \\
C_3 &= 4 \cdot 3 \cdot 2 \cdot 1 \cdot 1 = 24.
\end{align*}
The trivial part of the descent polynomial is
\[
T(t) = \frac{1}{5 \cdot 4 \cdot 3 \cdot 3 \cdot 2 \cdot 1 \cdot 1} \cdot t(t+2) = \frac{t(t+2)}{360}.
\]

Therefore, the descent polynomial is
\[
p(t) = \frac{t(t+2)}{360} \left( 27(t+5)(t+4)(t+3) + 27(t+5)(t+4) + 21(t+5) + 24 \right).
\]
To translate to the notation of \cite{diaz2017descent}, we simply use the shift $N = t+5$ to obtain
\[
d(\{3,5\};N) = \frac{(N-5)(N-3)}{360} \left( 27N(N-1)(N-2) + 27N(N-1) + 21N + 24 \right).
\]
For readers interested in computational data, we recommend Cai's work \cite{cai2021ratios}, especially the table of coefficients in the appendix.

\end{example}

\section{The Slice and Push Inequality}\label{sec:slice_push}

To prove Proposition~\ref{prop:ineq_coeff}, we apply an inductive argument to a slightly more general statement. For this, we consider a more general class of subdiagrams of a Young diagram $\Dbb(\lambda)$. Recall that we divide the diagram $\Dbb(\lambda)$ into diagonals $\ldots,X_{-1},X_0,X_1,\ldots$ by their contents.

Recall that a \emph{multiset} is, informally, a set in which each element may appear more than once. We say a multiset $D$ is a \emph{multi-subset} of a set $X$ if every element of $D$ is in $X$. For instance, $\{2,3,3\}$ is a multi-subset of $\{1,2,3\}$. If $D$ and $F$ are multisets, the \emph{multiset union} $D\sqcup F$ is the multiset where the multiplicity of each element is the sum of its multiplicities in $D$ and $F$; e.g., $\{2,3\}\sqcup\{3\}=\{2,3,3\}$. A \emph{subdiagram} of a diagram $\Dbb(\lambda)$ is a finite subset of cells of $\Dbb(\lambda)$. More generally, if $D$ is a finite multi-subset of cells of $\Dbb(\lambda)$, we call it a \emph{multi-subdiagram} of $\Dbb(\lambda)$. The \emph{weight} $\wt(D)$ of a multi-subdiagram is the product of the hook lengths of its cells taken with multiplicity. The following formula is easy to verify.

\begin{lemma}\label{lem:square_relation}
  If $F$ is any multi-subdiagram of $\Dbb(\lambda)$, and $\Dbb(\lambda)$ contains the collection of cells $\{c_{i,j},c_{i^{\pr},j},c_{i,j^{\pr}},c_{i^{\pr},j^{\pr}}\}$, then
  \[ \wt(F\sqcup\{c_{i,j}\})+\wt(F\sqcup\{c_{i^{\pr},j^{\pr}}\})\ =\ \wt(F\sqcup\{c_{i,j^{\pr}}\})+\wt(F\sqcup\{c_{i^{\pr},j}\}). \]  
\end{lemma}

We consider pairs $(D;F)$ where $D$ is a subdiagram and $F$ is a multi-subdiagram of $\Dbb(\lambda)$ such that for every cell $c_{i,j} \in D$, the cell $c_{i+1,j+1}$ exists in $\Dbb(\lambda)$. We refer to this pair as a \emph{$\sqci$-diagram}, depicted by labeling each cell in $D$ by a circle and each cell in $F$ by a square. The \emph{weight} of a $\sqci$-diagram $(D;F)$ is the sum of the weights of the multi-subdiagrams $D^{\pr}\sqcup F$, where the sum ranges over diagrams $D^{\pr}$ such that
\begin{itemize}
\item there exists a bijection $\eta: D\ra D^{\pr}$ with $\eta(c_{i,j})\in\{c_{i,j},\ c_{i+1,j+1}\}$, and
\item for each $k$, the restriction of $\eta$ to $D\cap(X_k\cup X_{k+1})$ is order-preserving.
\end{itemize}
In other words, one is allowed to move a circle up to one step southeast as long as it does not interfere with the cells from neighboring diagonals. We let $\wt(D;F)$ be the weight of the $\sqci$-diagram $(D;F)$.

\begin{figure}
  
  \centering
  \includegraphics{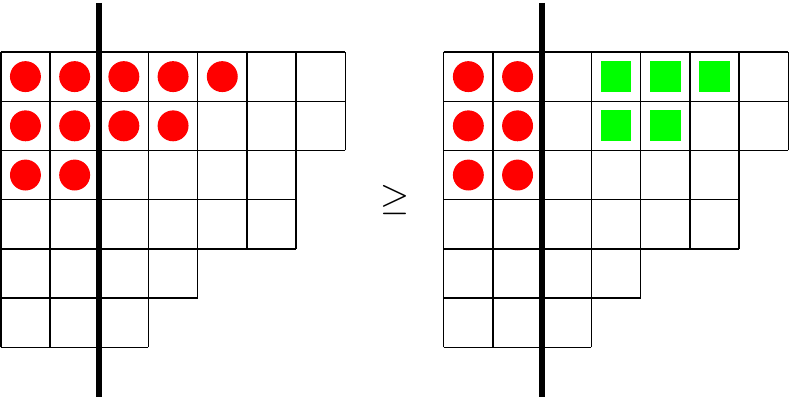}
  \caption{\label{fig:sap1}A depiction of the Slice and Push Inequality}
  
\end{figure}

We will use the following notation for constructing new diagrams from $D$ when $D$ is a subset of cells of $\Dbb(\lambda)$. We let $|_{k}D$ be the subdiagram of $D$ obtained by removing the first $k$ columns of $\Dbb(\lambda)$, while we let $D|_{k}$ be the subdiagram of $D$ contained in the first $k$ columns of $\Dbb(\lambda)$. We think of the bar $|_{k}$ as a ``knife'' placed between columns $k$ and $k+1$ where $D|_{k}$ is the portion of the diagram to the left and $|_{k}D$ is the portion to the right of the knife.

Similarly, we let $\ov{D}^i$ be the subdiagram obtained by removing the first $i$ rows and $\un{D}_i$ the subdiagram contained in the first $i$ rows of $\Dbb(\lambda)$. These subdiagrams are constructed by placing the knife horizontally instead of vertically.

We let $D^{\ra}$ be the diagram obtained by replacing each cell $c_{i,j}$ in $D$ with $c_{i,j+1}$; that is, we ``push'' every cell one step to the right. Similarly, $D^{\searrow}$ is the diagram with every cell pushed one step down and to the right.

For a diagram $D$, let $i_0$ (resp. $j_0$) be the first row (resp. column) occupied by at least one cell in $D$, and set
\[ D^{\circ} = \{ c_{i-i_0+1, j-j_0+1}\ |\ c_{i,j}\in D \}. \]
Then $D^{\circ}$ is the diagram obtained by translating $D$ as far north and west as possible while remaining inside $\Dbb(\lambda)$.

\begin{lemma}[Slice and Push Inequality]\label{lem:s_and_p}
  If $(D;F)$ is a $\sqci$-diagram such that $D^{\circ}=\Dbb(\mu)$ for some nonempty partition $\mu\subseteq\lambda$, then for any $k$,
  \begin{eqnarray}
    \wt(D; F) \geq \wt(D|_{k};\ F\sqcup (|_{k}D)^{\ra}).\label{ineq:sap}
  \end{eqnarray}

\end{lemma}

\begin{proof}
  We first observe that the multi-subdiagram $F$ contributes the same multiplicative factor to each side of the inequality, so we may assume $F=\emptyset$. We also assume that $|_{k}D$ only contains cells in a single column, namely the $(k+1)$-st column. To deduce the inequality for an arbitrary choice of $k$, we may iteratively slice off and push the rightmost column several times.

  Let $\mu$ be the partition for which $D^{\circ}=\Dbb(\mu)$. We proceed by induction on $|\mu|$. If $D$ only contains cells in column $k+1$, then
  \[ \wt(D;\emptyset) > \wt(\emptyset;D) > \wt(\emptyset;D^{\ra}) = \wt(D|_{k};\ (|_{k}D)^{\ra}). \]

  Suppose $i+1$ is the last row occupied by $D$ and $j$ is the first column occupied by $D$. We claim that (Figure~\ref{fig:sap2})
  \[ \wt(D;\emptyset)\ =\ \wt(|_{j}D;\ D|_{j})\ +\ \wt(\un{D}_{i};\ (\ov{D}^{i})^{\searrow}) \]
  Indeed, for any diagram $D^{\pr}$ in the sum on the left, the same diagram appears in one of the two summands on the right depending on whether the cell $c_{i+1,j}$ is in $D^{\pr}$. 
  Now assume that $\mu$ is \emph{not} a rectangle. This assumption ensures that the slices $|_{k}D$ and $\ov{D}^{i}$ are disjoint. By induction, we may apply the inequality (\ref{ineq:sap}) to each of the summands on the right to obtain:
  \begin{align*}
    \wt(D;\emptyset)\ &\geq \wt\left(|_{j}D|_{k};\ D|_{j}\sqcup(|_{k}D)^{\ra}\right)\\
    &\hphantom{\geq} +\wt\left(\un{(D|_{k})}_i;\ (\ov{D}^{i})^{\searrow}\sqcup(|_{k}D)^{\ra}\right)\\
    &= \wt(D|_{k};\ (|_{k}D)^{\ra}).
  \end{align*}

  \begin{figure}
    
    \centering
    \includegraphics{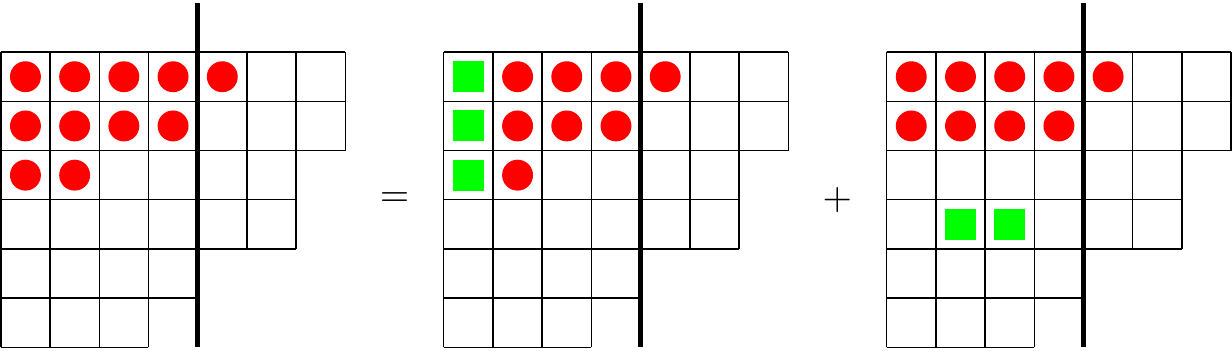}
    \caption{\label{fig:sap2}The non-rectangle case}
    
  \end{figure}

  Lastly, we assume that $\mu$ is a rectangle with at least two columns. Let $i$ (resp. $i^{\pr}$) be the first (resp. last) row occupied by a cell in $D$. Then by Lemma~\ref{lem:square_relation}, (first row of Figure~\ref{fig:sap3})
  \begin{align*}
    &\wt(D|_{k};\ (|_{k}D)^{\ra})\\
    &= \wt(D|_{k};\ (|_{k}D)^{\ra}\sqcup\{c_{i,k+1}\}\setm\{c_{i,k+2}\})\\
    &\hphantom{=}+ \wt(D|_{k};\ (|_{k}D)^{\ra}\sqcup\{c_{i^{\pr}+1,k+2}\}\setm\{c_{i,k+2}\})\\
    &\hphantom{=}- \wt(D|_{k};\ (|_{k}D)^{\ra}\sqcup\{c_{i^{\pr}+1,k+1}\}\setm\{c_{i,k+2}\})
  \end{align*}
  Applying a similar division as in the non-rectangle case, we have (second row of Figure~\ref{fig:sap3})
  \begin{align*}
    \wt(D;\emptyset) &= \wt\left(\ov{D}^i;\ \un{D}_i\right)\\
    &\hphantom{=}+\wt(D|_{k};\ (|_{k}D)^{\searrow})\\
    &= \wt\left(\ov{D}^i;\ \un{D}_i\right)\\
    &\hphantom{=}+\wt(D|_{k};\ (|_{k}D)^{\ra}\sqcup\{c_{i^{\pr}+1,k+2}\}\setm\{c_{i,k+2}\})
  \end{align*}
  In the same manner, (third row of Figure~\ref{fig:sap3})
  \begin{align*}
    &\wt(D|_{k};\ (|_{k}D)^{\ra}\sqcup\{c_{i,k+1}\}\setm\{c_{i,k+2}\})\\
    &= \wt\left(\ov{D|_{k}}^i;\ \un{D}_i\sqcup (\ov{|_{k}D}^i)^{\ra}\right)\\
    &\hphantom{=}+ \wt\left(D|_{k-1};\ (|_{k-1}D)^{\ra}\sqcup\{c_{i^{\pr}+1,k+1}\}\setm\{c_{i,k+2}\}\right)
  \end{align*}
  
  \begin{figure}

    \centering
    \includegraphics{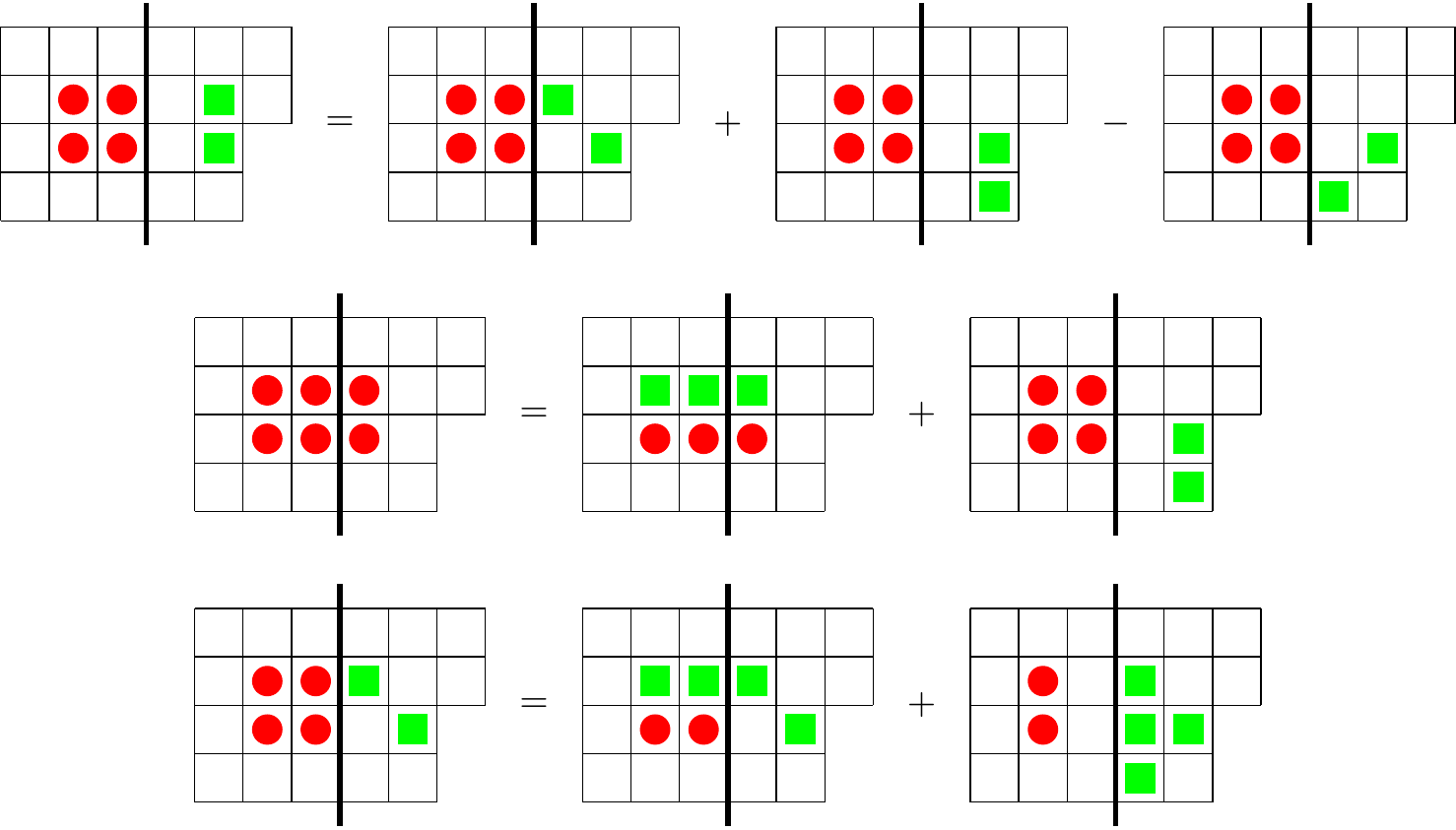}
    \caption{\label{fig:sap3}The rectangle case}
    
  \end{figure}

  Putting this together, we have
  \begin{align*}
    \wt(D;\emptyset) &- \wt(D|_{k};\ (|_{k}D)^{\ra})\\
    &= \wt\left(\ov{D}^i;\ \un{D}_i\right)\\
    &\hphantom{=}- \wt(D|_{k};\ (|_{k}D)^{\ra}\sqcup\{c_{i,k+1}\}\setm\{c_{i,k+2}\})\\
    &\hphantom{=}+ \wt(D|_{k};\ (|_{k}D)^{\ra}\sqcup\{c_{i^{\pr}+1,k+1}\}\setm\{c_{i,k+2}\})\\
    &= \wt\left(\ov{D}^i;\ \un{D}_i\right)\\
    &\hphantom{=}- \wt\left(\ov{D|_{k}}^i;\ \un{D}_i\sqcup (\ov{|_{k}D}^i)^{\ra}\right)\\
    &\hphantom{=}+ \wt(D|_{k};\ (|_{k}D)^{\ra}\sqcup\{c_{i^{\pr}+1,k+1}\}\setm\{c_{i,k+2}\})\\
    &\hphantom{=}- \wt\left(D|_{k-1};\ (|_{k-1}D)^{\ra}\sqcup\{c_{i^{\pr}+1,k+1}\}\setm\{c_{i,k+2}\}\right)
  \end{align*}

  In the last expression, the difference between the first two terms and the difference between the last two terms are both nonnegative by the inductive hypothesis. This completes the proof of the rectangle case.
\end{proof}

\begin{remark}
One may wonder whether the stronger inequality
\[\wt(D;F) \geq\wt(D|_{k};\ F\sqcup (|_{k}D))\]
 holds; i.e., if you ``slice'' but do not ``push.'' In fact, this inequality does not hold in general. For example, consider $\lambda = (4,3)$, $D = \{c_{1,1}, c_{1,2}\}$, $F = \emptyset$, and $k = 1$. We note that
\[
\wt(D; F) = 5 \cdot 4 + 5 \cdot 1 + 2 \cdot 1 = 27,
\]
while if $k=1$,
\[
\wt(D|_{k};\ F\sqcup (|_{k}D)) = 5 \cdot 4 + 2 \cdot 4 = 28.
\]
\end{remark}

\begin{example} \label{ex:sap_55432}

Here we give an illustration of the Slice and Push Inequality. Suppose that the partition $\lambda$ is $\lambda = (5,5,4,3,2)$. Consider the subset $D$ of cells of $\mathbb{D}(\lambda)$ given by $D = \{c_{1,1}, c_{1,2}, c_{2,1}, c_{2,2}\}$. In Figure \ref{9120ge4560} (left), we show the four cells of $D$ as the four circles. The set $D|_{1}$ consists of two cells: the two circles in Figure \ref{9120ge4560} (right). Once we slice $D$ by a knife between columns $1$ and $2$ into a left portion and a right portion, and then push the right portion one unit to the right, the moved right portion is now $\left(|_{1}D\right)^{\rightarrow}$. The set $\left(|_{1}D\right)^{\rightarrow}$ is shown as the two squares in Figure \ref{9120ge4560} (right).

The Slice and Push Inequality says that $\wt(D;\emptyset)$ is at least $\wt(D|_{1};\left(|_{1}D\right)^{\rightarrow})$; i.e., the weight of the $\sqci$-diagram on the left of Figure \ref{9120ge4560} is greater than or equal to the weight of the $\sqci$-diagram on the right.

Indeed, we can compute the two weights explicitly. We observe that $\wt(D;\emptyset) = 9 \cdot 8 \cdot 8 \cdot 7 + 9 \cdot 8 \cdot 8 \cdot 3 + 9 \cdot 8 \cdot 5 \cdot 3 + 9 \cdot 8 \cdot 5 \cdot 3 + 9 \cdot 5 \cdot 5 \cdot 3 + 7 \cdot 5 \cdot 5 \cdot 3 = 9120$, and that $\wt(D|_{1};\left(|_{1}D\right)^{\rightarrow}) = \left(9 \cdot 8 + 9 \cdot 5 + 7 \cdot 5 \right) \cdot 6 \cdot 5 = 4560$. In this example, the weight of the left diagram is exactly twice that of the right diagram.

\end{example}

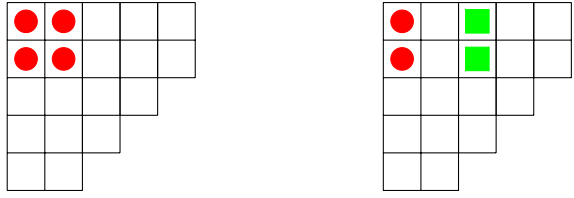
\begin{figure}
\begin{tikzpicture}[scale = 0.5]
\def\w{1} % width of a box.
\def\h{1} % height of a box.
\def\r{0.3} % radius of a circle.
\def\s{5/8} % side length of a square.
% Code to draw a circle.
\def\drawacircle[#1,#2,#3]{
\filldraw[#3] (#1 +\w/2, #2 + \h/2) circle (\r);
}
% Code to draw a square.
\def\drawasquare[#1,#2,#3]{
\filldraw[#3] (#1 +\w/2 -\s/2, #2 + \h/2 -\s/2) rectangle (#1 +\w/2 +\s/2, #2 + \h/2 +\s/2);
}
\begin{scope}[shift={(-5,0)}]

% Draw the Young diagram.
\foreach \xa\ya\xb\yb in {0/0/2/5,2/1/3/5,3/2/4/5,4/3/5/5}
{
  \draw[step = 1] (\xa*\w,\ya*\h) grid (\xb*\w,\yb*\h);
}

% Draw circles for the left diagram.
\foreach \x in {0,1}{
  \foreach \y in {3,4}{
    \drawacircle[\x,\y,red]
  }
}
\end{scope}
\begin{scope}[shift={(+5,0)}]

% Draw the Young diagram.
\foreach \xa\ya\xb\yb in {0/0/2/5,2/1/3/5,3/2/4/5,4/3/5/5}
{
  \draw[step = 1] (\xa*\w,\ya*\h) grid (\xb*\w,\yb*\h);
}

% Draw circles for the right diagram.
\foreach \y in {3,4}{
  \drawacircle[0,\y,red]
}

% Draw squares for the right diagram.
\foreach \y in {3,4}{
  \drawasquare[2,\y,green]
}
\end{scope}
\end{tikzpicture}
\caption{\label{9120ge4560} Two $\sqci$-diagrams on the Young diagram of the partition $\lambda = (5,5,4,3,2)$.}
\end{figure}

For the rest of this section, we let $D=\Dbb(\mu)$ where $\Dbb(\lambda/\mu)$ is a ribbon. The weight of $(D;\emptyset)$ is equal to the excitation factor of $f^{\lambda/\mu}$. Likewise, the coefficients that appear in the polynomial
\[ E(t) = \sum_{d=0}^sC_{s-d}\prod_{i=1}^d(t+\alpha_i) \]
are the weights of some $\sqci$-diagrams.

\begin{lemma}\label{lem:ribbon_s_and_p}
  \[ C_{s-i} = \wt(\ov{D}^1|_i;\ (|_i\ov{D}^1)^{\ra})\prod_{j=1}^{s-i}h(c_{\lambda_{i+j+1}^{\pr},i+j+1}) \]
\end{lemma}

\begin{proof}
  The quantity $C_{s-i}\prod_{j=1}^i(t+\alpha_j)$ is the sum of the weights of excited diagrams $D^{\pr}$ with $i$ cells in the first row. Dividing by the weights of the cells in the first row, we have
  \[ C_{s-i} = \wt(\ov{D}^1|_i;\ (|_i D)^{\searrow}). \]
We may rewrite the weight of the $\sqci$-diagram as follows.
  \begin{align*}
    C_{s-i} &= \wt(\ov{D}^1|_i;\ (|_i \ov{D}^1)^{\ra}\sqcup\{c_{\lambda_{i+2}^{\pr},i+2},\ \ldots,\ c_{\lambda_{s+1}^{\pr},s+1}\})\\
    &= \wt(\ov{D}^1|_i;\ (|_i \ov{D}^1)^{\ra})\cdot h(c_{\lambda_{i+2}^{\pr},i+2})\cdots h(c_{\lambda_{s+1}^{\pr},s+1})
  \end{align*}
This completes the proof.
\end{proof}

The hook length of a cell at the bottom of its column satisfies
\[ h(c_{\lambda_i^{\pr},i}) = \begin{cases} h(c_{\lambda_{i+1}^{\pr},i+1}) + 1\ \mbox{if }\lambda_i^{\pr}=\lambda_{i+1}^{\pr}\\ 1\ \mbox{otherwise.}\end{cases} \]
Here, we set $\lambda_i^{\pr}=0$ if $i>\lambda_1$.

\begin{corollary}\label{cor:refine_ineq_coeff}
  If $\lambda_{i+1}^{\pr}>\lambda_{i+2}^{\pr}$ for some $i \le s$, then
  \[ \frac{C_{s-i}}{0!} \geq \frac{C_{s-i+1}}{1!} \geq \cdots \geq \frac{C_s}{i!}. \]
\end{corollary}

\begin{proof}
Applying the Slice and Push Inequality with Lemma~\ref{lem:ribbon_s_and_p}, it follows that when $\ell<k$
\[ C_{s-\ell} \leq C_{s-k}\prod_{j=1}^{k-\ell}h(c_{\lambda^{\pr}_{\ell+j+1},\ell+j+1}). \]

  If $\lambda_{i+1}^{\pr}>\lambda_{i+2}^{\pr}$, then $h(c_{\lambda_{i+1}^{\pr},i+1})=1$. By induction, we have for $\ell<k\leq i$ that
\[ \prod_{j=1}^{k-\ell}h(c_{\lambda^{\pr}_{\ell+j+1},\ell+j+1}) \leq \frac{(i-\ell)!}{(i-k)!}. \]
This completes the proof.
\end{proof}

Proposition~\ref{prop:ineq_coeff} now follows from Corollary~\ref{cor:refine_ineq_coeff} by taking $i=s$. We note that by assumption, $\lambda_{s+2}^{\pr}=0$ and $\lambda_{s+1}^{\pr}\geq 2$, so the conditions of Corollary~\ref{cor:refine_ineq_coeff} are satisfied. We will also make use of the case in which $(|_i\ov{D}^1)$ is empty.

\begin{corollary}\label{cor:refine_eq_coeff}
  If $\lambda_{i+1}^{\pr}=\lambda_{s+1}^{\pr}=2$ for some $i\leq s$, then
  \[ \frac{C_0}{0!} = \frac{C_1}{1!} = \cdots = \frac{C_{s-i}}{(s-i)!}. \]
\end{corollary}

\begin{proof}
  By Lemma~\ref{lem:ribbon_s_and_p}, if $k\in\{0,1,\ldots,s-i\}$ then
  \begin{align*}
    C_{s-i-k} &= \wt(\ov{D}^1|_{i+k};\ (|_{i+k}\ov{D}^1)^{\ra})\prod_{j=1}^{s-i-k}h(c_{\lambda_{i+j+k+1}^{\pr},i+j+k+1})\\
    &= \wt(\ov{D}^1;\emptyset)\cdot (s-i-k)!
  \end{align*}
The corollary is proved.
\end{proof}

\section{Proof of the Conjecture of Diaz-Lopez, Harris, Insko, Omar, and Sagan}\label{sec:main_conj}

For the remainder of the paper, we fix a ribbon shape $\lambda/\mu$ with $\lambda_1=s+1,\ \mu_1=s$. Ribbons are connected diagrams, so we have $\lambda_{i+1} > \mu_i$ whenever $\mu_i \neq 0$. Let $\alpha_i=h(c_{1,i};\lambda)-1$ for all $i\in[s]$ and set $m=\alpha_1$. Let $E(t)$ be the polynomial function for the excitation factor of $\lambda^{(t)}/\mu$ with coefficient sequence $(C_0,\ldots,C_s)$ defined by
\[ E(t) = \sum_{d=0}^s C_{s-d}\prod_{i=1}^d(t+\alpha_i). \]
We let $z$ be a complex number such that $E(z) = 0$. Theorem~\ref{conj:main} is implied by the following two inequalities: $|z+m| \le m$ and $|z+1| \le m$. In Appendices~\ref{app:analytic} and~\ref{app:perturbation}, we will prove a complex analytic lemma and a polynomial perturbation lemma. In this section, we will use the Slice and Push Inequality and the two lemmas to prove the two desired inequalities.

\begin{theorem}\label{thm:main_left}
Let $z$ be a complex number such that $E(z) = 0$. Then, $|z+m| \le m$.
\end{theorem}

\begin{proof}
Recall that $\alpha_{s-i} \ge i+2$, for all $i = 0,1, \dots, s-1$. The case in which $\alpha_{s-i} = i+2$ for all $i$ was proved in \cite[Theorem 4.4]{diaz2017descent}. Let us suppose there exists some $i$ such that $\alpha_{s-i} \ge i+3$. Let $\kappa \ge 0$ be the smallest index such that $\alpha_{s-\kappa} \ge \kappa + 3$. By Corollary~\ref{cor:refine_eq_coeff},
\[
\frac{C_0}{0!} = \frac{C_1}{1!} = \cdots = \frac{C_{\kappa}}{\kappa!}.
\]
We argue that
\[
\frac{C_{\kappa+1}}{1!} \ge \frac{C_{\kappa+2}}{2!} \ge \cdots \ge \frac{C_s}{(s-\kappa)!}. \tag{$\ast$}
\]
If $\kappa = 0$, then ($\ast$) follows from Proposition~\ref{prop:ineq_coeff}. If $\kappa \ge 1$, then note that $\lambda'_{s - \kappa} > \lambda'_{s - \kappa + 1}$, and therefore, Corollary~\ref{cor:refine_ineq_coeff} gives
\[
\frac{C_{\kappa+1}}{0!} \ge \frac{C_{\kappa+2}}{1!} \ge \frac{C_{\kappa+3}}{2!} \ge \cdots \ge \frac{C_s}{(s-\kappa-1)!},
\]
which is a stronger inequality than ($\ast$).

If $z=0$ or $z=-\alpha_i$ for some $i$, it is easy to see that the desired inequality holds. Suppose $z \neq 0$ and $z+\alpha_i \neq 0$ for all $i$. We have
\begin{align*}
0 &= \frac{E(z)}{C_0\prod_{i=1}^s(z+\alpha_i)} \\
&=1 + \frac{1!}{z+2} + \frac{2!}{(z+2)(z+3)} + \cdots + \frac{\kappa!}{(z+2) \cdots (z+\kappa+1)} \\
&\hphantom{=}+ \frac{C_{\kappa+1}/C_0}{(z+2) \cdots (z+\kappa+1)} \cdot \Bigg( \frac{1}{z+\alpha_{s-\kappa}} + \frac{C_{\kappa+2}/C_{\kappa+1}}{(z+\alpha_{s-\kappa-1})(z+\alpha_{s-\kappa})} \\
&\hphantom{+ \frac{(\kappa+1)!}{(z+2) \cdots (z+\kappa+1)} \cdot \Bigg(}+ \cdots + \frac{C_s/C_{\kappa+1}}{(z+\alpha_1) \cdots (z+\alpha_{s-\kappa})} \Bigg).
\end{align*}
By induction, one may show that
\[
1 + \frac{1!}{z+2} + \cdots + \frac{\kappa!}{(z+2) \cdots (z+\kappa+1)} = \frac{(z+1)(z+2) \cdots (z+\kappa+1) - (\kappa+1)!}{z(z+2)(z+3) \cdots (z+\kappa+1)}.
\]
Inserting this into the previous equation gives:
\begin{align*}
&\left| \frac{(z+1)(z+2) \cdots (z+\kappa+1) - (\kappa+1)!}{z} \right| \\
&= \left| \frac{C_{\kappa+1}}{C_0}\Bigg( \frac{1}{z+\alpha_{s-\kappa}} + \cdots + \frac{C_s/C_{\kappa+1}}{(z+\alpha_1)\cdots(z+\alpha_{s-\kappa})} \Bigg) \right|.
\end{align*}

Using $C_{\kappa+1}/C_0\leq (\kappa+1)!$ along with ($\ast$) and the triangle inequality, we obtain
\begin{align*}
&\left| \frac{(z+1)(z+2) \cdots (z+\kappa+1) - (\kappa+1)!}{z(\kappa+1)!} \right| \\
&\le \frac{1!}{|z+\alpha_{s-\kappa}|} + \frac{2!}{|z+\alpha_{s-\kappa-1}| |z+\alpha_{s-\kappa}|} + \cdots + \frac{(s-\kappa)!}{|z+\alpha_1| |z+\alpha_2| \cdots |z+\alpha_{s-\kappa}|}.
\end{align*}

Suppose instead that $|z+m| > m$. Then, for $i = \kappa, \dots, s-1$, we have $|z+\alpha_{s-i}|\geq |z+m| - (m-\alpha_{s-i}) > i+3$. Therefore,
\begin{align*}
&\frac{1!}{|z+\alpha_{s-\kappa}|} + \frac{2!}{|z+\alpha_{s-\kappa-1}| |z+\alpha_{s-\kappa}|} + \cdots + \frac{(s-\kappa)!}{|z+\alpha_1| |z+\alpha_2| \cdots |z+\alpha_{s-\kappa}|} \\
&< \frac{1!}{\kappa+3} + \frac{2!}{(\kappa+3)(\kappa+4)} + \cdots + \frac{(s-\kappa)!}{(\kappa+3)(\kappa+4) \cdots (s+2)} < \frac{1}{\kappa+1}.
\end{align*}
On the other hand, $|z+(\kappa+2)| \ge |z+m| - (m - \kappa-2) > \kappa + 2$. Thus, Lemma~\ref{l:geoball} gives
\[
\left| \frac{(z+1)(z+2) \cdots (z+\kappa+1) - (\kappa+1)!}{z(\kappa+1)!} \right| \ge \frac{1}{\kappa+1},
\]
a contradiction. We have proved that $|z+m| \le m$.
\end{proof}

\begin{theorem}\label{thm:main_right}
Let $z$ be a complex number such that $E(z) = 0$. Then, $|z+1| \le m$.
\end{theorem}

\begin{proof}
When $s = 0$, the result holds vacuously. Assume $s \ge 1$. This theorem is a consequence of the Polynomial Perturbation Lemma (Lemma~\ref{l:pp}). We use the lemma for when $(g_1, g_2, \dots, g_k)$ is $\left( \frac{C_1}{C_0}, \frac{C_2}{C_1}, \dots, \frac{C_s}{C_{s-1}} \right)$ and $(a_1, a_2, \dots, a_k)$ is $(\alpha_s - 1, \alpha_{s-1}-1, \dots, \alpha_1 -1 )$. From Proposition~\ref{prop:ineq_coeff}, we have that for all $i$, $g_i \le i$, and therefore, the lemma applies. We have $|z+1| \le m$.
\end{proof}

\section*{Acknowledgments}

The first author gratefully acknowledges the support from the MIT Mathematics Department Strang Fellowship. The second author was supported by the National Science Foundation under Grant No. DMS-1440140 while in residence at the Mathematical Sciences Research Institute in Berkeley, California. We thank Bruce Sagan for reading the paper and providing many suggestions and corrections. The second author also thanks Bruce for introducing the conjecture to him, and he thanks Olivier Bernardi, Ira Gessel, and Bruce Sagan for their discussion with him at the Brandeis Combinatorics Seminar. We thank Wijit Yangjit for his discussion, which led to the proof of the Polynomial Perturbation Lemma. We thank Andrew Cai for insightful discussions, and for the easy description of the trivial part of the descent polynomial as shown in Subsection \ref{subsec:computing}. We thank Igor Pak for his suggestions on an earlier version of the paper. We used the programming languages python and R to generate examples and perform computations. We also used Wolfram Alpha to perform calculations.

\bibliographystyle{alpha}
\bibliography{bib_roots}{}

\appendix

\bigskip

\section{A Complex Analytic Lemma}\label{app:analytic}
In this section, we prove a lemma used in the proof of Theorem~\ref{thm:main_left}.

Let $k$ be a fixed positive integer. Consider the meromorphic function
\[
P(z) := \frac{(z+1)(z+2) \cdots (z+k) - k!}{z} \in \mathbb{C}[z].
\]

Since the numerator is divisible by $z$, the function $P(z)$ may be regarded as a polynomial function on the whole complex plane. We hereafter refer to $P(z)$ as a polynomial.

\begin{lemma}\label{lem:mero_roots}
  If $z_0$ is a root of $P(z)$, then $|z_0+1|\leq k$ and $|z_0+k|<k$.
\end{lemma}

\begin{proof}
  This proof is similar to that of \cite[Theorem 4.4]{diaz2017descent}.

  When $k=1$, we have $P(z)=1$ for all $z$, so we may assume $k>1$. The polynomial $P(z)$ has a nonzero constant term $\sum_{i=1}^k\frac{k!}{i}$, so $0$ is not a root of $P(z)$.

  If $|z_0+1|>k$ then $|z_0+i|>k+1-i$ for all $i\in\{1,\ldots,k\}$. This implies
\[
|(z_0+1)\cdots(z_0+k) - k!|\geq|(z_0+1)\cdots(z_0+k)|-k!>0,
\]
so $P(z_0)\neq 0$.

  Similarly, if $|z_0+k|>k$, then $|z_0+i|>i$ for all $i\in\{1,\ldots,k\}$. Again, we have $P(z_0)\neq 0$.

  Finally, if $|z_0+k|=k$ and $z_0\neq 0$ then $|z_0+i|>i$ for $i\in\{1,\ldots,k-1\}$, and we again deduce that $P(z_0)$ is nonzero.
\end{proof}

Lemma~\ref{lem:mero_roots} implies that $\frac{1}{P(z)}$ is holomorphic in the domain $\{z \in \mathbb{C} : |z+(k+1)| > k+1\}$. The \emph{maximum modulus principle} states that the modulus of any non-constant holomorphic function does not have a local maximum in any open, connected domain. As $\frac{1}{P(z)}$ vanishes at infinity, the maximum value of $\frac{1}{|P(z)|}$ in the domain $\{z:\ |z+(k+1)|\geq k+1\}$ is attained in the compact subset $\{z:\ M\geq |z+(k+1)|\geq k+1\}$ for any sufficiently large value of $M$. Hence, this maximum value is achieved at the boundary $\{z:\ |z+(k+1)|=k+1\}$.

\begin{lemma} \label{l:geoball}
If $z$ is a complex number such that $|z+(k+1)| \ge k+1$, then $|P(z)| \ge (k-1)!$.
\end{lemma}

It is easy to see that the inequality holds when $z = 0$. From now on, assume $z \neq 0$. Let us suppose $k \ge 8$ and deal with small values of $k$ later. As stated above, we can assume $|z+(k+1)| = k+1$.

Without loss of generality, assume $\Im(z)\ge 0$ and write $z = -a+bi$, where $a, b \ge 0$. Note that $(k+1)^2 - |z|^2 = (k+1-a)^2 - a^2$, so
\begin{equation}\label{eq:a}
a = \frac{|z|^2}{2(k+1)}.  
\end{equation}

We consider two cases.

\underline{Case 1.} $|z| \ge \frac{2}{k}$.

We have
\begin{align*}
\left| (z+1) \cdots (z+k) \right| &= \sqrt{ \prod_{j=1}^k \left( (j-a)^2 + b^2 \right) } \\
&= k! \cdot \sqrt{\prod_{j=1}^k \left( \left( 1 - \frac{a}{j} \right)^2 + \left( \frac{b}{j} \right)^2 \right)}.
\end{align*}
Note that $\left( 1 - \frac{a}{j} \right)^2 + \left( \frac{b}{j} \right)^2 > 1$ because $|z+j| > j$. For $x_1, \dots, x_k > 0$, we have the inequality $\prod_{j=1}^k (1+x_j) > 1 + \sum_{j=1}^k x_j$. Using this inequality with $k \ge 8$, $|z| \ge \frac{2}{k}$, and \eqref{eq:a}, we have
\begin{align*}
\left| (z+1) \cdots (z+k) \right| &> k! \cdot \sqrt{1+ \sum_{j=1}^k \left[ - \frac{2a}{j} + \frac{a^2+b^2}{j^2} \right]} \\
&> k! \cdot \sqrt{1 - \left( \frac{\log k + 1}{k+1} \right)|z|^2 + \frac{3}{2} |z|^2} \\
&> k! \cdot \sqrt{1 + \frac{2|z|}{k} + \frac{|z|^2}{k^2}} \\
&= \left( 1 + \frac{|z|}{k} \right) \cdot k!.
\end{align*}
Therefore, $|(z+1) \cdots (z+k) - k!| \ge |(z+1) \cdots (z+k)| - k! > |z| \cdot (k-1)!$.

\smallskip

\underline{Case 2.} $|z| < \frac{2}{k}$.

We claim that
\begin{equation}\label{eq:case2}
\frac{9999}{10000} |z| \le b \le a+b \le \frac{73}{72} |z| < \frac{73}{36k}.
\end{equation}

The last inequality is immediate from the hypothesis on $|z|$. Note that
\[
a = \frac{|z|^2}{2(k+1)} \le \frac{2}{k} \cdot \frac{|z|}{2(k+1)} = \frac{|z|}{k(k+1)} \le \frac{|z|}{72}.
\]
Therefore, $a+b \le \frac{|z|}{72} + |z| = \frac{73}{72} |z|$. We also have
\[
b = \sqrt{|z|^2 - a^2} = |z| \cdot \sqrt{1 - \frac{|z|^2}{4(k+1)^2}} \ge |z| \cdot \sqrt{1 - \frac{1}{(k(k+1))^2}} \ge \frac{9999}{10000} |z|,
\]
as claimed.

The following useful trigonometric inequalities can be verified by single variable calculus:
\begin{itemize}
\item For all $x \ge 0$,
\[
x-x^2 \le \arctan(x) \le x. \tag{T1}
\]
\item For all $x \ge 0$,
\[
x - \frac{x^3}{3} \le \sin(x). \tag{T2}
\]
\end{itemize}

For each $j = 1,2, \dots, k$, let $ \theta_j := \arg(z+j) \in [0,2\pi)$. Let $\theta = \theta_1 + \cdots + \theta_k$. We have that
\[
(z+1) \cdots (z+k) = | (z+1) \cdots (z+k) | \cdot e^{i \theta}.
\]

Now, we claim that
\[
(\log k) b \le \theta \le (1+ \log k) (a+b) < \frac{\pi}{4}.
\]

Since $k\geq 8$, \eqref{eq:case2} implies that $a+b<1$, from which it follows that
\[
\frac{b}{j-a} \le \frac{a+b}{j}
\]
whenever $j\geq 1$. Then for $j = 1, 2, \dots, k$,
\[
\theta_j = \arctan \left( \frac{b}{j-a} \right) \le \arctan \left( \frac{a+b}{j} \right).
\]
Therefore,
\[
\theta \le \sum_{j=1}^k \arctan \left( \frac{a+b}{j} \right) \overset{\text{(T1)}}\le \sum_{j=1}^k \frac{a+b}{j} \le (1+\log k)(a+b).
\]
On the other hand,
\[
\theta_j = \arctan \left( \frac{b}{j-a} \right) \ge \arctan \left( \frac{b}{j} \right).
\]
Hence,
\[
\theta \ge \sum_{j=1}^k \arctan \left( \frac{b}{j} \right) \overset{\text{(T1)}}\ge \sum_{j=1}^k \left( \frac{b}{j} - \frac{b^2}{j^2} \right) \ge \left( \frac{1}{2} + \log k \right) b - \frac{\pi^2}{6} b^2 \ge (\log k) b.
\]
Since $k\geq 8$, we also have
\[
(1+\log k)(a+b) \le (1+ \log k) \cdot \frac{73}{36 k} < \frac{\pi}{4},
\]
as claimed. We remark that the last inequality does not hold for $k=7$.

The function $x \mapsto x - \frac{x^3}{3}$ is increasing on $\left[ 0, \frac{\pi}{4} \right]$. Thus,
\[
\sin \theta \overset{\text{(T2)}}{\ge} \theta - \frac{\theta^3}{3} \ge (\log k)b - \frac{(\log k)^3 b^3}{3} \ge \left( \log k - \frac{4(\log k)^3}{3k^2} \right) \frac{9999}{10000} |z| \ge |z|.
\]
For the last inequality, we used the fact that $k\geq 8$. Therefore,
\begin{align*}
|(z+1) \cdots (z+k) - k!| &\ge \Im \left( (z+1) \cdots (z+k) - k! \right) \\
&= |(z+1) \cdots (z+k)| \cdot \sin \theta \ge k! \cdot |z| \ge (k-1)! \cdot |z|.
\end{align*}
This finishes the proof of the inequality for $k \ge 8$.

\medskip

Finally, we consider the cases with $k\leq 7$. Since the seven cases for $k$ may all be proved in roughly the same manner, we give a proof for $k=7$ and leave the other cases to the reader.

Let $k=7$ and $w = z + 5$. We have $|w| \ge 5$.
\begin{align*}
&\left| \frac{(z+1)(z+2)(z+3)(z+4)(z+5)(z+6)(z+7) - 7!}{z} \right| \\
&=\left| (z+5)^6 - 2(z+5)^5 - 3(z+5)^4 + 20(z+5)^3 + 44(z+5)^2 + 192(z+5) + 1008 \right|\\
&= \left| w^6 - 2w^5 - 3w^4 + 20w^3 + 44w^2 + 192w + 1008 \right| \\
&\ge |w|^6 - 2|w|^5 - 3|w|^4 - 20|w|^3 - 44|w|^2 - 192|w| - 1008 \ge 1932 > 6!.
\end{align*}

In each of the remaining cases, one may use a similar argument where $w$ is defined as in the following table.

\vspace{2mm}

\begin{center}
\begin{tabular}{c|c}
  $k$ & $w$\\\hline\hline
  $1$ & $z$\\\hline
  $2$ & $z+3$\\\hline
  $3$ & $z+3$\\\hline
  $4$ & $z+4$\\\hline
  $5$ & $z+4$\\\hline
  $6$ & $z+\frac{17}{4}$
\end{tabular}
\end{center}

\medskip

\section{The Polynomial Perturbation Lemma}\label{app:perturbation}
In this section, we prove a lemma on polynomial perturbation. Starting with a certain polynomial with distinct real roots, we obtain a new polynomial by perturbing it in a certain bounded manner. The lemma gives an upper bound on the moduli of the roots of the resulting polynomial. We also note that our result is sharp.

\begin{lemma} {\em (Polynomial Perturbation)} \label{l:pp}
Suppose that $k$ is a positive integer. Let $a_1 < a_2 < \cdots < a_k$ be a strictly increasing sequence of positive integers. Let $g_1, \dots, g_k \ge 0$ satisfy $g_i \le i$ for all $i=1, \dots, k$. Define
\begin{align*}
P(z) &:= (z+a_k) (z+a_{k-1}) \cdots (z+a_1) + g_1 (z+a_k) (z+a_{k-1}) \cdots (z+a_2) \\
&\hphantom{:= } + g_1 g_2 (z+a_k)(z+a_{k-1}) \cdots (z+a_3) + \cdots + g_1 g_2 \cdots g_{k-1} (z+a_k) + g_1 g_2 \cdots g_k.
\end{align*}
If $z$ is a complex root of $P(z)$, then $|z| \le a_k + 1$.
\end{lemma}

We think of the first term $(z+a_k) \cdots (z+a_1)$ as the {\em main} term and the rest as the perturbation. The original roots of the main term are $-a_k, \dots, -a_1$, which all lie inside the closed ball $\{z: |z| \le a_k \}$. The lemma says that the roots of the perturbed polynomial are inside a slightly larger closed ball $\{z: |z| \le a_k + 1\}$.

To show this lemma, we prove the following stronger statement. This strategy is predictable as we have the picture that the main term should dominate the rest.
\begin{claim}
Let $a_1, \dots, a_k, g_1, \dots, g_k$ be as in the above lemma. If $|z| > a_k + 1$, then
\begin{equation}\label{eq:claim}
\left| (z+a_k) \cdots (z+a_1) \right| > \left| \sum_{i=1}^k g_1 \cdots g_i (z+a_{i+1}) \cdots (z+a_k) \right|.
\end{equation}
\end{claim}

We will first assume that $k \ge 3$, and then work on $k=1, 2$ later.

\subsection*{Step I. Reduction of $g_i$'s.}
In the first step, we will reduce the problem to the case in which $g_i = i$ for all $i$. For this purpose, we consider $z$ and $a_1, \dots, a_k$ fixed within this step. Define
\[
F \left(g_1, \dots, g_k \right) := \left| \sum_{i=1}^k g_1 \cdots g_i (z+a_{i+1}) \cdots (z+a_k) \right|.
\]
Note that $F$ is a convex function for each $g_i \in [0,i]$. Therefore,
\begin{align*}
&F \left(g_1, \dots, g_{i-1}, g_i, g_{i+1}, \dots, g_k \right) \\
&\le \max \left\{ F \left( g_1, \dots, g_{i-1}, 0, g_{i+1}, \dots, g_k \right), F \left(g_1, \dots, g_{i-1}, i, g_{i+1}, \dots, g_k \right)\right\}.
\end{align*}
Using the inequality above for all $i \in [k]$, we find that there exist $\widehat{g}_1, \dots, \widehat{g}_k \in \mathbb{R}$ such that
\[
F(g_1, \dots, g_k) \le F \left( \widehat{g}_1, \dots, \widehat{g}_k \right),
\]
where $\widehat{g}_i \in \{0,i\}$ for each $i = 1,2, \dots, k$. If $\widehat{g}_1 = 0$, then $F\left( \widehat{g}_1, \dots, \widehat{g}_k \right) = 0$ and \eqref{eq:claim} follows immediately. Suppose that $\widehat{g}_1 = 1$. Let $k' \in [k]$ be the largest index such that $\widehat{g}_i = i$ for all $1 \le i \le k'$. By the definition of $F$, we note that
\[
F \left( \widehat{g}_1, \dots, \widehat{g}_k \right) = F \left(1, 2, \dots, k', 0, \dots, 0 \right).
\]
It suffices to show
\[
\left| (z+a_{k'}) \cdots (z+a_1) \right| > \left| \sum_{i=1}^{k'} i! (z+a_{i+1}) \cdots (z+a_{k'}) \right|.
\]
Since $\{z: |z| > a_k + 1\} \subseteq \{z: |z| > a_{k'} + 1\}$, it suffices to prove the claim for the case where $k$ is replaced by $k'$ and $g_i = i$ for all $i \in [k']$.

From now on we assume $g_i = i$ for all $i = 1, 2, \dots, k$.

\medskip

\subsection*{Step II. Reduction of $a_i$'s.}
We consider the vector of first differences of $a_i$'s
\[
\Delta := \left( a_k - a_{k-1}, \dots, a_3 - a_2, a_2 - a_1 \right) \in \left( \mathbb{Z}_{>0} \right)^{k-1}.
\]
Define
\[
\mathcal{A} := \left\{ (1^{k-1}), (1^{k-2}, 2), (1^{k-2}, 3), (1^{k-2}, 4), (1^{k-3}, 2, 1), (1^{k-3}, 2, 2) \right\}.
\]
Here, $1^m$ denotes $m$ copies of $1$'s for each $m \ge 0$. The goal of this step is to prove the claim when $\Delta \notin \mathcal{A}$.

Assume $\Delta \notin \mathcal{A}$. Let $z \in \mathbb{C}$ such that $|z| > a_k + 1$. There are two cases.

\underline{Case 1.} Suppose that there is some index $j$ such that $4 \le j \le k$ and $a_j - a_{j-1} \ge 2$. In particular, we must have $k \ge 4$ for this case to happen. By the triangle inequality, 
\begin{equation}\label{eq:apBineq}
|z+a_i| > k+1 - i  
\end{equation}
holds for $i = 1, 2, \dots, k$. In this case, we obtain better inequalities for $i=1,2,3$: $|z+a_1| > k+1$, $|z+a_2| > k$, and $|z+a_3| > k-1$.

Recall that we want to show that
\[
|(z+a_k) \cdots (z+a_1)| > \left| \sum_{i=1}^k i! (z+a_{i+1}) \cdots (z+a_k) \right|.
\]
That is,
\[
1 > \left| \sum_{i=1}^k \frac{i!}{(z+a_1) \cdots (z+a_i)} \right|.
\]
By the triangle inequality, it suffices to prove that
\[
\sum_{i=1}^k \frac{i!}{|z+a_1| \cdots |z+a_i|} < 1.
\]
Using the bounds for $|z+a_i|$ in \eqref{eq:apBineq}, we obtain
\begin{align*}
&\sum_{i=1}^k \frac{i!}{|z+a_1| \cdots |z+a_i|} \\
&< \frac{1!}{k+1} + \frac{2!}{(k+1)k} + \frac{k-2}{k+1} \left( \frac{3!}{k(k-1)(k-2)} + \frac{4!}{k(k-1)(k-2)(k-3)} + \cdots + \frac{k!}{k!} \right) \\
&= \frac{2}{k^2 - 1} + \frac{k-2}{k+1} \sum_{i=1}^k \binom{k}{i}^{-1} < 1,
\end{align*}
as desired. To see why the last inequality holds, one may use the bound
\[
\sum_{i=1}^k \binom{k}{i}^{-1} < 1 + \frac{3}{k},
\]
for $k \ge 8$, and check the cases where $4 \le k \le 7$ by hand.

\underline{Case 2.} Suppose that $a_i - a_{i-1} = 1$ for all $i \ge 4$. In this case, because $\Delta \notin \mathcal{A}$, we have
\[
(a_3 - a_2, a_2 - a_1) \notin \left\{ (1,1), (1,2), (1,3), (1,4), (2,1), (2,2) \right\}.
\]
Note that $A := |z+a_1| > k-2 + (a_3 - a_1)$ and $B := |z+a_2| > k-2 + (a_3 - a_2)$. By some easy casework, we obtain the bound $(A-1)B > k^2+2k-3$. With the same argument as in the case above, we want to show that
\[
\sum_{i=1}^k \frac{i!}{|z+a_1| \cdots |z+a_i|} < 1.
\]
Observe that
\[
\sum_{i=1}^k \frac{i!}{|z+a_1| \cdots |z+a_i|} < \frac{1}{A} + \frac{k(k-1)}{AB} \sum_{i=2}^k \binom{k}{i}^{-1}.
\]
Thus, it suffices to show
\[
k(k-1) \sum_{i=2}^k \binom{k}{i}^{-1} \le (A-1)B.
\]
Using $(A-1)B > k^2+2k-3$ and $\sum_{i=2}^k \binom{k}{i}^{-1} \le 1 + \frac{21}{10 k}$, we obtain the desired inequality.

\medskip

\subsection*{Step III. The case in which $\Delta \in \mathcal{A}$.}
Still assuming that $k \ge 3$ is a fixed integer, we work on the six remaining cases of $\Delta$.

\underline{Case 1.} $a_3 - a_2 = a_2 - a_1 = 1$. In this case, $(a_k, \dots, a_1) = (k + \alpha, k-1 + \alpha, \dots, 1 + \alpha)$ for some integer $\alpha \ge 0$. Note that if we prove the desired claim for the case when $\alpha = 0$, all other cases when $\alpha > 0$ will follow. Thus, we may now assume $(a_k, \dots, a_1) = (k, k-1, \dots, 1)$. The inequality we want to prove becomes
\begin{equation} \label{eq:sharp}
|(z+1) (z+2) \cdots (z+k)| > \left| \sum_{i=1}^k i! (z+i+1)(z+i+2) \cdots (z+k) \right|.
\end{equation}
By induction, we can show that
\[
\sum_{i=0}^k i! (z+i+1)(z+i+2) \cdots (z+k) = \frac{z(z+1)(z+2) \cdots (z+k) - (k+1)!}{z-1}.
\]
Therefore, the inequality (\ref{eq:sharp}) is equivalent to
\[
\left| \frac{1}{z-1} - \frac{(k+1)!}{(z-1)(z+1)(z+2) \cdots (z+k)} \right| < 1.
\]
By the triangle inequality, it suffices to show that
\begin{equation} \label{eq:sharp2}
\frac{1}{|z-1|} \left( 1 + \frac{(k+1)!}{|z+1| |z+2| \cdots |z+k|} \right) < 1.
\end{equation}
Moreover, by the triangle inequality, we know that $|z+j| \ge |z|-j > k+1-j$, for $j = 1, 2, \dots, k$, and also $|z-1| \ge |z| - 1 > k$.

If $|z-1| \ge k+2$, it is easy to see that the bound (\ref{eq:sharp2}) holds. Assume $|z-1| < k+2$. Let $w = z + k + 1$. Write $w = a + bi$, where $a, b \in \mathbb{R}$. From $|z-1| < k+2$, we have $(a-(k+2))^2 + b^2 < (k+2)^2$, and so
\[
a > \frac{a^2+b^2}{2(k+2)}.
\]
In particular, we have $a > 0$.

From $|z| > k+1$, we have $(a-(k+1))^2 + b^2 > (k+1)^2$, and so
\begin{equation} \label{eq:(a^2+b^2)/(2a)}
\frac{a^2+b^2}{2a} > k + 1.
\end{equation}
The inequality (\ref{eq:sharp2}) is equivalent to
\[
|z+1| |z+2| \cdots |z+k| + (k+1)! < |z-1| |z+1| |z+2| \cdots |z+k|,
\]
which is
\begin{align}
&\left| 1 - \frac{w}{k} \right| \left| 1 - \frac{w}{k-1} \right| \cdots | 1 - w | + (k+1) \notag \\
&< (k+2) \left| 1 - \frac{w}{k+2} \right| \left| 1 - \frac{w}{k} \right| \left| 1 - \frac{w}{k-1} \right| \cdots | 1 - w |. \label{eq:sharp3}
\end{align}
Note that for $j = 1, 2, \dots, k$, we have $\left| 1 - \frac{w}{j} \right| > 1$, while $\left| 1 - \frac{w}{k+2} \right| < 1$. Therefore, we can write
\[
\prod_{j=1}^k \left| 1 - \frac{w}{j} \right| = 1 + \varepsilon,
\]
and $\left| 1 - \frac{w}{k+2} \right| = 1 - \varepsilon'$, where $\varepsilon, \varepsilon' > 0$. Now, (\ref{eq:sharp3}) is equivalent to
\begin{equation} \label{eq:sharp4}
\varepsilon' < \frac{k+1}{k+2} \cdot \frac{\varepsilon}{1+\varepsilon}.
\end{equation}
From $|z-1| > k$, we know $\varepsilon' < \frac{2}{k+2}$. If $\varepsilon \ge \frac{2}{k-1}$, then (\ref{eq:sharp4}) is clear. Assume $\varepsilon < \frac{2}{k-1}$. Thus, $\frac{\varepsilon}{1+\varepsilon} > \frac{k-1}{k+1} \cdot \varepsilon$. To show (\ref{eq:sharp4}), it suffices to show that
\begin{equation} \label{eq:sharp5}
\frac{\varepsilon}{\varepsilon'} \ge \frac{k+2}{k-1}.
\end{equation}
Using a trick similar to one in the proof of Lemma A.2, we find
\[
(1+\varepsilon)^2 = \prod_{j=1}^k \left( 1 - \frac{2a}{j} + \frac{a^2+b^2}{j^2} \right) \ge 1 + \sum_{j=1}^k \left( - \frac{2a}{j} + \frac{a^2+b^2}{j^2} \right).
\]
This shows
\begin{equation} \label{eq:eps(2+eps)}
\varepsilon(2+\varepsilon) \ge \sum_{j=1}^k \left( - \frac{2a}{j} + \frac{a^2+b^2}{j^2} \right).
\end{equation}
On the other hand, since $(1-\varepsilon')^2 = \left| 1 - \frac{w}{k+2} \right|^2 = \left( 1 - \frac{a}{k+2} \right)^2 + \left( \frac{b}{k+2} \right)^2$, we have
\begin{equation} \label{eq:eps'(2-eps')}
\varepsilon'(2-\varepsilon') = \frac{2a}{k+2} - \frac{a^2+b^2}{(k+2)^2}.
\end{equation}
Therefore, we have
\begin{align*}
&\frac{\varepsilon(2+\varepsilon) - 5 \varepsilon'(2-\varepsilon')}{2a} \\
&\overset{\text{(\ref{eq:eps(2+eps)}), (\ref{eq:eps'(2-eps')})}}\ge \left( 1 + \frac{1}{2^2} + \frac{1}{3^2} + \cdots + \frac{1}{k^2} + \frac{5}{(k+2)^2} \right) \frac{a^2+b^2}{2a} - \left( 1 + \frac{1}{2} + \frac{1}{3} + \cdots + \frac{1}{k} + \frac{5}{k+2} \right) \\
&\overset{\text{(\ref{eq:(a^2+b^2)/(2a)})}}> \left( 1 + \frac{1}{2^2} + \frac{1}{3^2} + \cdots + \frac{1}{k^2} + \frac{5}{(k+2)^2} \right) (k+1) - \left( 1 + \frac{1}{2} + \frac{1}{3} + \cdots + \frac{1}{k} + \frac{5}{k+2} \right) > 0.
\end{align*}
This shows that $\varepsilon(2+\varepsilon) > 5 \varepsilon'(2-\varepsilon')$, and thus
\[
\frac{\varepsilon}{\varepsilon'} > 5 \cdot \frac{2-\varepsilon'}{2+\varepsilon} > 5 \cdot \frac{2 - \frac{2}{k+2}}{2+\frac{2}{k-1}} = 5 \cdot \frac{(k+1)(k-1)}{k(k+2)} > \frac{k+2}{k-1}.
\]
It is easy to see that the last inequality $5 \cdot \frac{(k+1)(k-1)}{k(k+2)} > \frac{k+2}{k-1}$ holds, for $k \ge 3$. This proves (\ref{eq:sharp5}) and we have finished the proof for this case.

\underline{Case 2.} $a_3 - a_2 = 1$ and $d := a_2 - a_1 \in \{2,3,4\}$. Write
\[
(a_k, \dots, a_1) = (u+k+1, u+k, \dots, u+3, u+3-d),
\]
where $u \ge 0$ is an integer. The inequality we want to prove becomes
\begin{align*}
&|(z+u+k+1)(z+u+k) \cdots (z+u+3)(z+u+3-d)| \\
&> \left| \sum_{i=1}^k i! (z+u+k+1) (z+u+k) \cdots (z+u+i+2) \right|.
\end{align*}
Note that
\begin{align*}
&\sum_{i=1}^k i! (z+u+k+1) (z+u+k) \cdots (z+u+i+2) \\
&= \frac{(z+u+k+1)(z+u+k) \cdots (z+u+2) - (k+1)!}{z+u}.
\end{align*}
Therefore, it suffices to show that
\begin{equation} \label{eq:1_1234}
|z+u| > \left| 1 + \frac{d-1}{z+u+3-d} - \frac{(k+1)!}{(z+u+k+1) \cdots (z+u+3)(z+u+3-d)} \right|.
\end{equation}
Recall that $|z| > a_k + 1 = u+k+2$. Therefore, by the triangle inequality, we have
\begin{align*}
RHS_{\text{(\ref{eq:1_1234})}} &\le 1 + \frac{d-1}{|z+u+3-d|} + \frac{(k+1)!}{|z+u+k+1| \cdots |z+u+3| |z+u+3-d|} \\
&< 1 + \frac{d-1}{k+d-1} + \frac{(k+1)!}{(k-1)!(k+d-1)} \\
&= 1 + \frac{d-1}{k+d-1} + \frac{k(k+1)}{k+d-1} \\
&\le k+2 < |z+u| = LHS_{\text{(\ref{eq:1_1234})}}.
\end{align*}

\underline{Case 3.} $a_3 - a_2 = 2$ and $d := a_2 - a_1 \in \{1,2\}$. Write
\[
(a_k, \dots, a_1) = (u+k+1, u+k, \dots, u+4, u+2, u+2-d),
\]
where $u \ge 0$ is an integer. The inequality we want to prove becomes
\begin{align*}
&|(z+u+k+1)(z+u+k) \cdots (z+u+4)(z+u+2)(z+u+2-d)| \\
&> \Bigg| 1! (z+u+k+1)(z+u+k) \cdots (z+u+4)(z+u+2) \\
&\hphantom{> \Bigg|}+ \sum_{i=2}^k i! (z+u+k+1)(z+u+k) \cdots (z+u+i+2) \Bigg|.
\end{align*}
Note that
\begin{align*}
&1! (z+u+k+1)(z+u+k) \cdots (z+u+4)(z+u+2) \\
&+ \sum_{i=2}^k i! (z+u+k+1)(z+u+k) \cdots (z+u+i+2) \\
&= \frac{(z+u+k+1) (z+u+k) \cdots (z+u+4) \left( (z+u)^2 + 4(z+u) + 6 \right) - (k+1)!}{z+u}.
\end{align*}
Therefore, it suffices to show that
\begin{align}
|z+u| &> \Bigg| 1 + \frac{d}{z+u+2-d} + \frac{2}{(z+u+2)(z+u+2-d)} \label{eq:2_12} \\
&\hphantom{> \Bigg|} - \frac{(k+1)!}{(z+u+k+1)(z+u+k) \cdots (z+u+4)(z+u+2)(z+u+2-d)} \Bigg|. \notag
\end{align}
Again, by the triangle inequality, we have
\begin{align*}
RHS_{\text{(\ref{eq:2_12})}} &\le 1 + \frac{d}{|z+u+2-d|} + \frac{2}{|z+u+2| |z+u+2-d|} \\
&\hphantom{\le} + \frac{(k+1)!}{|(z+u+k+1)(z+u+k) \cdots (z+u+4)(z+u+2)(z+u+2-d)|} \\
&< 1 + \frac{d}{k+d} + \frac{2}{k(k+d)} + \frac{(k+1)!}{(k-2)!k(k+d)} \\
&= 1 + \frac{2}{k(k+d)} + \frac{k^2+d-1}{k+d} \\
&\le k+2 < |z+u| = LHS_{\text{(\ref{eq:2_12})}}.
\end{align*}

This finishes Step III.

\medskip

\subsection*{Step IV. Small $k$.}
Finally, we work with the cases $k=1,2$.

\fbox{$k=1$} Suppose $|z| > a_1 + 1$. Then, $|z+a_1| \ge |z| - a_1 > 1 \ge |g_1|$.

\fbox{$k=2$} For this case, we prove one little lemma.

\begin{lemma}
Let $u,v \in \mathbb{R}_{\ge 0}$ such that $u \ge v+1$. For any $z \in \mathbb{C}$ with $|z| > u + 1$, we have
\[
|z+u| |z+v| > |z+u+2|.
\]
\end{lemma}
\begin{proof}
By the triangle inequality, we have $|z+v| > (u+1)-v \ge 2$. By classical geometry, the locus of all points $\zeta \in \mathbb{C}$ satisfying
\[
2 \, |\zeta+u| \le |\zeta+u+2|
\]
is the closed disk enclosed by the circle of Apollonius centered at $-u + \frac{2}{3}$ of radius $\frac{4}{3}$. Since the whole of this disk lies inside $\{\zeta: |\zeta| \le u+1\}$, we finish the proof.
\end{proof}

Applying the lemma for $(u,v) = (a_2, a_1)$, we find that for $z \in \mathbb{C}$ such that $|z| > a_2 + 1$,
\[
|z+a_2| |z+a_1| > |z+a_2+2|.
\]
Since $|z+a_1| > 2$, we also have $|z+a_2| |z+a_1| > |z+a_2|$. Therefore,
\[
|z+a_2| |z+a_1| > \max \{|z+a_2+2|,|z+a_2| \} \ge |z+a_2+g_2| \ge |g_1(z+a_2) + g_1g_2| ,
\]
by the convexity argument as we did earlier. This finishes the proof.

This also concludes our proof of the Polynomial Perturbation Lemma.

\medskip

It is worth noting that the Polynomial Perturbation Lemma is sharp. When $k$ is odd, $\Delta = (1^{k-1})$, and $g_i = i$ for all $i \in [k]$, observe that $- a_k - 1$ is a root of $P(z)$. This implies that the upper bound on the moduli of the roots of the perturbed polynomial can indeed be attained.

\bigskip
\end{document}